\documentclass[A4,12pt]{article}
\usepackage{cite}
\usepackage{amsmath,amssymb,mathtools}
\usepackage{graphicx,framed,enumerate}
\usepackage{amsthm,thmtools,thm-restate}
\usepackage[margin=1in]{geometry}
\usepackage{fancyhdr}
\usepackage{tikz}

\usetikzlibrary{calc}

\usepackage{expl3}
\usepackage{xparse,adjustbox,ifthen}
\usepackage{calc,intcalc}
\usepackage{etoolbox}

\newcommand{\z}{\mathbb{Z}}

\renewcommand{\c}{\mathbb{C}}
\renewcommand{\k}{\mathbb{K}}

\renewcommand{\a}{\alpha}
\renewcommand{\b}{\beta}
\newcommand{\g}{\gamma}

\renewcommand{\l}{\lambda}

\newcommand{\<}{\left\langle}
\renewcommand{\>}{\right\rangle}

\newcommand{\is}{\cong}

\DeclareMathOperator{\cyc}{cyc}
\DeclareMathOperator{\Conj}{Conj}

\DeclareMathOperator{\Hom}{Hom}

\DeclareMathOperator{\Fix}{Fix}

\newcommand{\ds}{\displaystyle}
\newcommand{\ts}{\textstyle}

\newcommand{\pvec}[2]{ \ds\begin{bmatrix} {#1}\\{#2} \end{bmatrix} }

\newcommand{\minus}{\smallsetminus}

\renewcommand{\ss}{\subseteq}

\newtheorem{thm}{Theorem}[section]
\newtheorem{lem}[thm]{Lemma}
\newtheorem{prop}[thm]{Proposition}
\newtheorem{cor}[thm]{Corollary}

\newtheorem{question}[thm]{Question}
\theoremstyle{definition}
\newtheorem*{defn}{Definition}
\newtheorem*{exA}{Example A}
\theoremstyle{remark}
\newtheorem*{remark}{Remark}

\theoremstyle{empty}

\newcommand{\goodpair}{$C$-admissible}
\newcommand{\CRT}{the Chinese Remainder Theorem}

\begin{document}

\title{Parity-Unimodality and a Cyclic Sieving Phenomenon for Necklaces}
\author{Eric Stucky}
\date{22 September 2020}

\maketitle

\begin{abstract}
We discuss two surprising properties of a family of polynomials that generalize the Mahonian $q$-Catalan polynomials, and more generally the $q$-Schr\"oder polynomials. By interpreting them as $\mathfrak{sl}_2$-characters, we show that the rational $q$-Schr\"oder polynomials are parity-unimodal, which means that the even- and odd-degree coefficients are separately unimodal. Second, we show that they exhibit a $q=-1$ phenomenon. This is a special case of a more general cyclic sieving phenomenon for certain transitive $S_n$-actions, deduced from Molien's formula.
\end{abstract}

\section{Introduction}

Given a sequence $\alpha=(\alpha_1,\ldots,\alpha_r)$ of nonnegative integers that sums to $n$, the \textbf{multinomial coefficient}
$$ \binom{n}{\alpha}  =  \binom{n}{\alpha_1,\ldots,\alpha_r}  =  \frac{ n! }{ \alpha_1!  \cdots \alpha_r! } $$
is a positive integer, counting the number of words having exactly $\alpha_i$ occurrences of the letter $i$ for each $i=1,2,\ldots,r$.   The symmetric group $S_n$ acts on the set of such words by permuting positions, and when one restricts this action to the cyclic subgroup $C=\langle c \rangle$ generated by the $n$-cycle $c=(1,2,\ldots,n)$, the orbits are called \textbf{necklaces} with $\alpha_i$ \textbf{beads} of color $i$; we refer to these as $\alpha$-necklaces. It is easily seen that the $C$-action on $\alpha$-necklaces will be free if and only if $\gcd(\alpha)=\gcd(\alpha_1,\ldots,\alpha_r)=1$,  and thus the number of $\alpha$-necklaces in this case is given by  
$C(\alpha)=\frac{1}{n} \binom{n}{\alpha}$.

When $\alpha=(a,a+1)$, this is the well-known \textbf{Catalan number}:
$$ C(a,a+1)=\frac{1}{2a+1}\binom{2a+1}{a}=\frac{1}{a+1}\binom{2a}{a}. $$
For example, when $\alpha=(3,4)$, there are $C(3,4)=\frac{1}{7}\binom{7}{3}=\frac{1}{4}\binom{6}{3}=5$ such necklaces with $3$ black beads and 4 white beads, shown here:

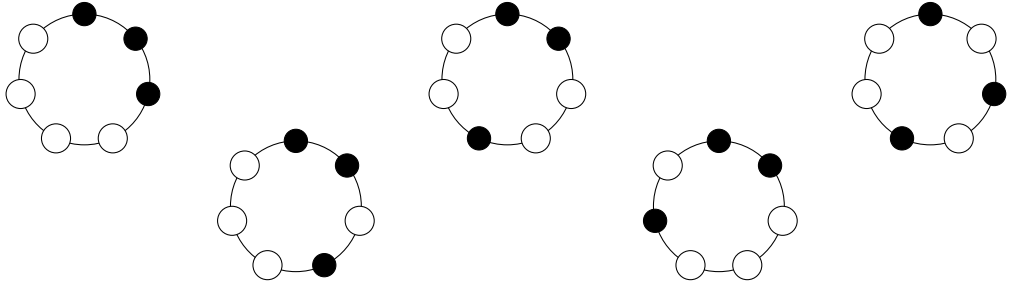
\vspace{0.2in}
\begin{center}
\begin{tikzpicture}[scale=1]
\pgfmathsetlengthmacro{\dot}{5.5pt};
\pgfmathsetlengthmacro{\layerheight}{48pt};
\pgfmathsetlengthmacro{\heightoffset}{\layerheight / 2};

\pgfmathsetmacro{\b}{3};
\pgfmathsetmacro{\a}{4};
\pgfmathsetlengthmacro{\sqsize}{8pt};

\pgfmathsetlengthmacro{\neckR}{\dot*4.5};
\pgfmathsetmacro{\k}{0};
\pgfmathsetmacro{\beads}{\a+\b-\k};
\pgfmathsetmacro{\beadSep}{360/\beads};
\pgfmathsetmacro{\beadOff}{90};
\pgfmathsetlengthmacro{\beadR}{\dot};

\coordinate (top1) at (0,\heightoffset+3*\layerheight);
\coordinate (top2) at (0,\heightoffset+2*\layerheight);
\coordinate (top3) at (0,\heightoffset+1*\layerheight);
\coordinate (TL) at (0pt,\heightoffset);
\coordinate (TR) at (80pt,\heightoffset);
\coordinate (BL) at (0pt,-\heightoffset);
\coordinate (BR) at (80pt,-\heightoffset);
\coordinate (bot3) at (0,-\heightoffset-1*\layerheight);
\coordinate (bot2) at (0,-\heightoffset-2*\layerheight);
\coordinate (bot1) at (0,-\heightoffset-3*\layerheight);

\begin{scope}[shift={(0pt,\heightoffset)}]
	\coordinate (center) at (0, 0);
	\draw [black,] (center) circle [radius=\neckR];
	\draw [black, fill=black] (center)+(\beadOff-0*\beadSep:\neckR) circle [radius=\beadR*0.8];
	\draw [black, fill=black] (center)+(\beadOff-1*\beadSep:\neckR) circle [radius=\beadR*0.8];
	\draw [black, fill=black] (center)+(\beadOff-2*\beadSep:\neckR) circle [radius=\beadR*0.8];
	\draw [black, fill=white, thin] (center)+(\beadOff-3*\beadSep:\neckR) circle [radius=\beadR];
	\draw [black, fill=white, thin] (center)+(\beadOff-4*\beadSep:\neckR) circle [radius=\beadR];
	\draw [black, fill=white, thin] (center)+(\beadOff-5*\beadSep:\neckR) circle [radius=\beadR];
	\draw [black, fill=white, thin] (center)+(\beadOff-6*\beadSep:\neckR) circle [radius=\beadR];
\end{scope}

\begin{scope}[shift={(160pt,\heightoffset)}]
	\coordinate (center) at (0, 0);
	\draw [black,] (center) circle [radius=\neckR];
	\draw [black, fill=black] (center)+(\beadOff-0*\beadSep:\neckR) circle [radius=\beadR*0.8];
	\draw [black, fill=black] (center)+(\beadOff-1*\beadSep:\neckR) circle [radius=\beadR*0.8];
	\draw [black, fill=black] (center)+(\beadOff-4*\beadSep:\neckR) circle [radius=\beadR*0.8];
	\draw [black, fill=white, thin] (center)+(\beadOff-2*\beadSep:\neckR) circle [radius=\beadR];
	\draw [black, fill=white, thin] (center)+(\beadOff-3*\beadSep:\neckR) circle [radius=\beadR];
	\draw [black, fill=white, thin] (center)+(\beadOff-5*\beadSep:\neckR) circle [radius=\beadR];
	\draw [black, fill=white, thin] (center)+(\beadOff-6*\beadSep:\neckR) circle [radius=\beadR];
\end{scope}

\begin{scope}[shift={(240pt,-\heightoffset)}]
	\coordinate (center) at (0, 0);
	\draw [black,] (center) circle [radius=\neckR];
	\draw [black, fill=black] (center)+(\beadOff-0*\beadSep:\neckR) circle [radius=\beadR*0.8];
	\draw [black, fill=black] (center)+(\beadOff-1*\beadSep:\neckR) circle [radius=\beadR*0.8];
	\draw [black, fill=black] (center)+(\beadOff-5*\beadSep:\neckR) circle [radius=\beadR*0.8];
	\draw [black, fill=white, thin] (center)+(\beadOff-2*\beadSep:\neckR) circle [radius=\beadR];
	\draw [black, fill=white, thin] (center)+(\beadOff-3*\beadSep:\neckR) circle [radius=\beadR];
	\draw [black, fill=white, thin] (center)+(\beadOff-4*\beadSep:\neckR) circle [radius=\beadR];
	\draw [black, fill=white, thin] (center)+(\beadOff-6*\beadSep:\neckR) circle [radius=\beadR];
\end{scope}

\begin{scope}[shift={(80pt,-\heightoffset)}]
	\coordinate (center) at (0, 0);
	\draw [black,] (center) circle [radius=\neckR];
	\draw [black, fill=black] (center)+(\beadOff-0*\beadSep:\neckR) circle [radius=\beadR*0.8];
	\draw [black, fill=black] (center)+(\beadOff-1*\beadSep:\neckR) circle [radius=\beadR*0.8];
	\draw [black, fill=black] (center)+(\beadOff-3*\beadSep:\neckR) circle [radius=\beadR*0.8];
	\draw [black, fill=white, thin] (center)+(\beadOff-2*\beadSep:\neckR) circle [radius=\beadR];
	\draw [black, fill=white, thin] (center)+(\beadOff-4*\beadSep:\neckR) circle [radius=\beadR];
	\draw [black, fill=white, thin] (center)+(\beadOff-5*\beadSep:\neckR) circle [radius=\beadR];
	\draw [black, fill=white, thin] (center)+(\beadOff-6*\beadSep:\neckR) circle [radius=\beadR];
\end{scope}

\begin{scope}[shift={(320pt,\heightoffset)}]
	\coordinate (center) at (0, 0);
	\draw [black,] (center) circle [radius=\neckR];
	\draw [black, fill=black] (center)+(\beadOff-0*\beadSep:\neckR) circle [radius=\beadR*0.8];
	\draw [black, fill=black] (center)+(\beadOff-2*\beadSep:\neckR) circle [radius=\beadR*0.8];
	\draw [black, fill=black] (center)+(\beadOff-4*\beadSep:\neckR) circle [radius=\beadR*0.8];
	\draw [black, fill=white, thin] (center)+(\beadOff-1*\beadSep:\neckR) circle [radius=\beadR];
	\draw [black, fill=white, thin] (center)+(\beadOff-3*\beadSep:\neckR) circle [radius=\beadR];
	\draw [black, fill=white, thin] (center)+(\beadOff-5*\beadSep:\neckR) circle [radius=\beadR];
	\draw [black, fill=white, thin] (center)+(\beadOff-6*\beadSep:\neckR) circle [radius=\beadR];
\end{scope}
\end{tikzpicture}
\end{center}

This paper concerns two surprising properties of a $q$-analogue of $C(\alpha)$:
\begin{equation}
\label{main-character}
C(\alpha;q) =\frac{1}{[n]_q} \pvec{n}{\alpha}_{q},
\end{equation}
defined in terms of these standard $q$-analogues:
\begin{align*}
\pvec{n}{\alpha}_{q} &=\frac{ [n]!_q }{ [\alpha_1]!_q  \cdots [\alpha_r]!_q }, \\
[n]!_q &=[n]_q [n-1]_q \cdots [2]_q [1]_q, \\
[n]_q &=1+q+q^2+\cdots+q^{n-1}.
\end{align*}

Notice that $C(k,a-k,b-k;q)$ agrees with the usual definition of the \textbf{rational $q$-Schr\"oder polynomial}; this specializes to a \textbf{rational $q$-Catalan polynomial} when $k=0$, and further specializes to \textbf{MacMahon's $q$-Catalan polynomial} when also $b=a+1$.

\subsection{Parity Unimodality}

Let us say that a polynomial $X(q)=\sum_i a_i q^i$ in $q$ with nonnegative coefficients $a_i$ is \textbf{parity-unimodal} if both subsequences $(a_0,a_2,a_4,\ldots)$ and $(a_1,a_3,a_5,\ldots)$ are unimodal. We have tested the following conjecture numerically up to $n\leq 30$:

\begin{restatable}{conj}{parityuniconj}
\label{parityuniconj}
When $\gcd(\alpha)=1$, the polynomial $C(\alpha;q)$ is parity-unimodal.
\end{restatable}

This conjecture appears to be difficult, but we can prove it for a particularly significant collection of $\alpha$. Namely, we will explain in Section \ref{necklaces} why known results in the theory of rational Cherednik algebras imply Conjecture \ref{parityuniconj} for $\alpha=(k,a-k,b-k)$ when $\gcd(a,b)=1$ and $0 \leq k \leq a$.

\begin{restatable}{thm}{parityunimodal}
\label{parityunimodal}
Let $a,b$, and $k$ be positive integers satisfying $\gcd(a,b)=1$ and $0 \leq k \leq a<b$. Then the rational $q$-Schr\"oder polynomial $C(k,a-k,b-k;q)$ is parity-unimodal.
\end{restatable}

\begin{remark}
The interested reader may find a growing body of related research on this subject. We wish to highlight to highlight two works that have been written while this paper was in preparation. First, Xin and Zhong \cite[Conjectures 3 and 4]{XinZho2020} prove Theorem \ref{parityunimodal} in the rational Catalan case ($k=0$) for small values of $a$. Although their work currently does not yield the full theorem, it has the advantage of being considerably more elementary than ours. Second, a different conjectural generalization of the Catalan case was given by Billey, Konvalinka, and Swanson \cite[Conjecture 4.3]{BilKonSwa2020} that is related to the major index statistic for standard Young tableaux.
\end{remark}

\subsection{Cyclic Sieving}
\label{necklacecsp}

Recall from Reiner, Stanton, and White \cite{ReiStaWhi2004} that for a set $X$ carrying the action of a cyclic group $\langle \tau \rangle$ of order $m$, and a polynomial $X(q)$ with nonnegative integer coefficients, one says that $(X,X(q),\langle \tau\rangle)$ exhibits the \textbf{cyclic sieving phenomenon} if for every integer $b$ one has that $ \left|\{ x \in X: \tau^b(x)=x \}\right| = \left[ X(q) \right]_{q=\zeta^b}$, where $\zeta=e^{\frac{2 \pi i}{m}}$.  

We are motivated by the case in which $m=2$, so that $\tau$ is an involution; that is,
\begin{align*}
X(1) &= |X|,\\
X(-1)&=\left|\{x \in X: \tau(x)=x\}\right|.
\end{align*}
In this case, $(X,X(q),\tau)$ is said to exhibit \textbf{Stembridge's $q=-1$ phenomenon} \cite{Ste1994}.

This lets us phrase our first result, which follows on the observation in \cite[\S 8]{ReiStaWhi2004} that whenever $\gcd(\alpha)=1$, the $q$-analogue $C(\alpha;q)$ defined in \eqref{main-character} is a polynomial in $q$ with nonnegative coefficients.  As noted above, $C(\alpha;1)=C(\alpha)$ counts the set $X$ of all of $\alpha$-necklaces.  There is a natural involutive action $\tau_0$ on $X$ which reflects a necklace over a line; orbits for this $\tau_0$-action are called \textbf{bracelets}. We say that a bracelet is \textbf{asymmetric} if it is a $\tau_0$-orbit of necklaces of size two.

\begin{restatable}{thm}{braceletcount}
\label{braceletcount}
When $\gcd(\alpha)=1$, the set $X$ of $\alpha$-necklaces along with $X(q):=C(\alpha;q)=\sum_i a_i q^i$ and its $\tau_0$-action by reflection exhibits Stembridge's $q=-1$ phenomenon. That is, 
\begin{align*}
\frac{1}{2}\left( C(\alpha;1) +C(\alpha;-1) \right)&= a_0+a_2+a_4+\cdots, \quad\text{and}\\
\frac{1}{2}\left( C(\alpha;1) -C(\alpha;-1) \right)&= a_1+a_3+a_5+\cdots,
\end{align*}
respectively count the total number of bracelets, and the number of asymmetric bracelets.
\end{restatable}

In the example of $\alpha=(3,4)$, one has 
$$ C(\alpha;q)=\frac{1}{[7]_q}\pvec{7}{3}_{q}=1+q^2+q^3+q^4+q^6, $$ 
with $\frac{1}{2}\left( C(\alpha;1) +C(\alpha;-1) \right)=4$ and $\frac{1}{2}\left( C(\alpha;1) +C(\alpha;-1) \right)=1$. This agrees with the fact that the five necklaces shown above give rise to four bracelets, only one of which is asymmetric, namely the bracelet shown here:

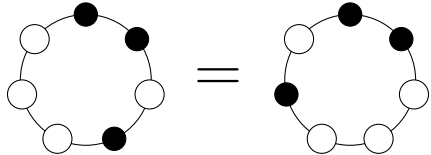
\begin{center}
\begin{tikzpicture}[scale=1]
\pgfmathsetlengthmacro{\dot}{5.5pt};
\pgfmathsetlengthmacro{\layerheight}{48pt};
\pgfmathsetlengthmacro{\heightoffset}{\layerheight / 2};

\pgfmathsetmacro{\b}{3};
\pgfmathsetmacro{\a}{4};
\pgfmathsetlengthmacro{\sqsize}{8pt};

\pgfmathsetlengthmacro{\neckR}{\dot*4.5};
\pgfmathsetmacro{\k}{0};
\pgfmathsetmacro{\beads}{\a+\b-\k};
\pgfmathsetmacro{\beadSep}{360/\beads};
\pgfmathsetmacro{\beadOff}{90};
\pgfmathsetlengthmacro{\beadR}{\dot};

\coordinate (top1) at (0,\heightoffset+3*\layerheight);
\coordinate (top2) at (0,\heightoffset+2*\layerheight);
\coordinate (top3) at (0,\heightoffset+1*\layerheight);
\coordinate (TL) at (0pt,\heightoffset);
\coordinate (TR) at (80pt,\heightoffset);
\coordinate (BL) at (0pt,-\heightoffset);
\coordinate (BR) at (80pt,-\heightoffset);
\coordinate (bot3) at (0,-\heightoffset-1*\layerheight);
\coordinate (bot2) at (0,-\heightoffset-2*\layerheight);
\coordinate (bot1) at (0,-\heightoffset-3*\layerheight);

\begin{scope}[shift={(50pt,-\heightoffset)}]
	\coordinate (center) at (0, 0);
	\draw [black,] (center) circle [radius=\neckR];
	\draw [black, fill=black] (center)+(\beadOff-0*\beadSep:\neckR) circle [radius=\beadR*0.8];
	\draw [black, fill=black] (center)+(\beadOff-1*\beadSep:\neckR) circle [radius=\beadR*0.8];
	\draw [black, fill=black] (center)+(\beadOff-5*\beadSep:\neckR) circle [radius=\beadR*0.8];
	\draw [black, fill=white, thin] (center)+(\beadOff-2*\beadSep:\neckR) circle [radius=\beadR];
	\draw [black, fill=white, thin] (center)+(\beadOff-3*\beadSep:\neckR) circle [radius=\beadR];
	\draw [black, fill=white, thin] (center)+(\beadOff-4*\beadSep:\neckR) circle [radius=\beadR];
	\draw [black, fill=white, thin] (center)+(\beadOff-6*\beadSep:\neckR) circle [radius=\beadR];
\end{scope}

\node at (0,-\heightoffset) {\color{black} \Huge $=$};

\begin{scope}[shift={(-50pt,-\heightoffset)}]
	\coordinate (center) at (0, 0);
	\draw [black,] (center) circle [radius=\neckR];
	\draw [black, fill=black] (center)+(\beadOff-0*\beadSep:\neckR) circle [radius=\beadR*0.8];
	\draw [black, fill=black] (center)+(\beadOff-1*\beadSep:\neckR) circle [radius=\beadR*0.8];
	\draw [black, fill=black] (center)+(\beadOff-3*\beadSep:\neckR) circle [radius=\beadR*0.8];
	\draw [black, fill=white, thin] (center)+(\beadOff-2*\beadSep:\neckR) circle [radius=\beadR];
	\draw [black, fill=white, thin] (center)+(\beadOff-4*\beadSep:\neckR) circle [radius=\beadR];
	\draw [black, fill=white, thin] (center)+(\beadOff-5*\beadSep:\neckR) circle [radius=\beadR];
	\draw [black, fill=white, thin] (center)+(\beadOff-6*\beadSep:\neckR) circle [radius=\beadR];
\end{scope}

\end{tikzpicture}
\end{center}

Theorem \ref{braceletcount} will be deduced in Section \ref{gpthy}  from a much more general statement. Notice that the reflection $\tau_0$, thought of as an element of $S_n$, is contained in the normalizer of $C$. In particular, we provides a sufficient condition for other $\tau\in N_{S_n}(C)$ acting on $\a$-necklaces to satisfy a cyclic sieving phenomenon as well. Finally, in Section \ref{futurework} we suggest that even this more general statement is just one instantiation of a much broader framework that we call \textbf{secondary cyclic sieving}.

\section{Proof of Theorem \ref{parityunimodal}}
\label{necklaces}

We begin by fixing some notation. Any $d$-dimensional $G$-representation $V$ gives rise to a \textbf{symmetric algebra} $SV^*$; in coordinates, $SV^*$ is simply a polynomial ring $\c[x_1,\dots, x_d]$, where the variables $x_i$ are basis vectors for $V^*$. The action is given in the natural way: $g\cdot p(x_1,\dots, x_d)=p(g^{-1}x_1,\dots, g^{-1}x_d)$. Note that $SV^*$ is a graded vector space, and each graded piece $S^iV^*$ is a $G$-representation.

The representation $V$ also gives rise to an \textbf{exterior representation} $\wedge V$, which in particular is a graded vector space. Let $e_1,\dots, e_d$ be a basis for $V$; then for $1\leq k\leq d$, the $k^\text{th}$ graded component $\wedge^k V$ is spanned by $e_{i_1}\wedge \cdots \wedge e_{i_k}$, subject to the relations $e_i \wedge e_j = -(e_j\wedge e_i)$, and the multilinearity of $\wedge$. This space inherits the $G$-representation from $V$ diagonally, that is:
$$ g\cdot (e_{i_1}\wedge \cdots \wedge e_{i_k}) = g(e_{i_1})\wedge \cdots \wedge g(e_{i_k}).$$

We also recall the definition of a Hilbert series. For any graded vector space $V=\bigoplus_{i\in\z} V_i$, its \textbf{Hilbert series} is the formal Laurent series $\mathcal H(V;q)=\sum_{i\in\z}\dim(V_i)q^i$. Frequently our vector spaces will be positively graded and finite dimensional, in which case the Hilbert series is simply a polynomial in $q$. 

For the remainder of the section, we let $V=\c^{a-1}$ be the irreducible reflection representation of $S_a$. When the $\otimes$ symbol appears without a subscript, it always means the tensor product of complex vector spaces (or their elements).

\subsection{A Model for the Rational $q$-Schr\"oder Polynomials}

We approach parity-unimodality by defining a notion of $q$-Schr\"oder numbers which is \emph{a priori} different from any $C(\alpha; q)$. The existence of these polynomials is rooted in a ``BGG resolution'' for rational Cherednik algebra modules, which for present purposes can be phrased in more elementary language:

\begin{thm}
\label{hsop}
Let $U\ss SV^*$ be an $(a-1)$-dimensional $S_a$-subrepresentation contained in degree $b$ and denote by $\<U\>$ the ideal generated by the elements of $U$. If $SV^*/\!\<U\>$ is a finite-dimensional vector space, then as a graded $SV^*$-module and $\c[S_a]$-module, it admits a resolution
$$ 0 \leftarrow SV^*/\!\<U\> \leftarrow SV^* \leftarrow \Big(SV^*\otimes {\wedge}^1U\Big)  \leftarrow\cdots\leftarrow \Big(SV^*\otimes {\wedge}^{a-1}U\Big) \leftarrow 0.$$
\end{thm}

This result is not new. An early proof of a similar statement given by Berest, Etingof, and Ginzburg \cite[Theorem 2.4]{BerEtiGin2003}, which has since been generalized considerably (for instance, \cite[\S 3]{EtiMa2010}) Still, for the sake of completeness, we sketch the proof here.

\begin{proof}
Let $\theta_1,\dots,\theta_{a-1}$ be a basis for $U$. Observe that $SV^*\otimes U$ is a free $SV^*$-module (by multiplication on the left) with basis $\{1\otimes\theta_i\}$; where there is no risk of confusion we will also call this basis $\{\theta_i\}$. The Koszul complex $K(\theta_1,\dots, \theta_{a-1})$ over the ring $SV^*$ begins with the canonical quotient of $SV^*$ onto the quotient, and the higher terms involve exterior powers of $U$:

$$ 0 \leftarrow SV^*/\!\<U\> \leftarrow SV^*\leftarrow \Big(SV^*\otimes {\wedge}^1U\Big)  \leftarrow\cdots\leftarrow \Big(SV^*\otimes {\wedge}^{a-1}U\Big) \leftarrow 0.$$
The leftmost nonzero map in this complex is the canonical quotient, and the others are the usual $SV^*$-linear differentials:
$$ d_k: \theta_{i_1}\wedge \cdots \wedge \theta_{i_k} ~\mapsto~ \sum_{j=1}^k (-1)^{k+1}\iota(\theta_{i_j})\cdot \theta_1\wedge\cdots\wedge \widehat{\theta_{i_j}}\wedge\cdots\wedge \theta_{i_k}$$
where the hat denotes a factor which is omitted, and $\iota$ is the canonical inclusion $U\to SV^*$.

As usual, the Koszul complex will be exact, and hence a resolution of graded $SV^*$-modules, if $\theta_1,\dots,\theta_{a-1}$ is a maximal $SV^*$-regular sequence. This is true because $SV^*$ is Cohen-Macaulay and $\theta_1,\dots,\theta_{a-1}$ is a homogeneous system of parameters, which in turn follows from the fact that $SV^*/\!\<U\>$ is finite-dimensional.

To complete the proof, we need to show that the maps in the Koszul complex are $S_a$-equivariant. This follows directly from the definitions after straightforward, if somewhat lengthy, calculations.
\end{proof}

Theorem \ref{hsop} is a conditional result, computing a resolution when provided with a ``nice'' $S_a$-representation $U$. Dunkl proved that if $b$ is coprime to $a$, then such a $U$ does actually exist, and moreover it is essentially unique, with $U\is V^*$ as $S_a$-representations. We will record the necessary details in the next subsection as Theorem \ref{hsopsexist}; until then, the skeptical reader may regard the following results as conditional as well.

Although explicit formulas for the $\theta_i$ are tricky, the resulting quotient space $SV^*/\!\<U\>$ is well-studied. For instance, it is the space of ``rational parking functions'' as defined by Armstrong, Loehr, and Warrington \cite{ArmLoeWar2016}. In the following subsection we will introduce the rational Cherednik algebra, and it is true that $SV^*/\!\<U\>$ is irreducible as a module over this algebra (see, for instance, Chmutova and Etingof \cite{ChmEti2003}). In the latter context it is often called $L_{b/a}(1)$; we adopt this notation here.

We will not be interested in $L_{b/a}(1)$ per se, but rather the intertwiners between it and the exterior algebra of the defining (permutation) representation $\c^a$. Precisely, for $\gcd(a,b)=1$ and $0\leq k\leq a<b$, the \textbf{rational $q$-Schr\"oder numbers} $C^k_{a,b}(q)$ are defined to be a normalized Hilbert series:
$$ q^{\binom{k}{2}} C^k_{a,b}(q) = \mathcal H\Big( \!\Hom_{\Bbb C [S_a]}\!\big(\wedge^k\!\c^a
, L_{b/a}(1)\big) ; \, q\Big). $$

\begin{prop}
\label{schroderformula}
Let $a,b$, and $k$ be nonnegative integers satisfying $\gcd(a,b)=1$ and $k \leq a<b$. Then the rational $q$-Schr\"oder number $C^k_{a,b}(q)$ coincides with the rational $q$-Schr\"oder polynomial $C(k, a-k, b-k; q)$.
\end{prop}

The proof is primarily a computation using Theorem \ref{hsop} together with \cite[Theorem 1]{KirPak1990}, corrected and simplified by Molchanov in \cite{Mol1992}. Recall that a list of nonnegative integers $\lambda=(\l_1,\l_2,\dots, \l_r)$ that sums to $n$ is said to be a \textbf{partition} of $n$, written $\l\vdash n$, if it is also non-increasing: $\l_1\geq \l_2\geq\cdots\geq \l_r$. A foundational result in the representation theory of $S_n$ is that partitions of $n$ are in  explicit bijective correspondence with irreducible representations of $S_n$. 

We write $S^\lambda$ for the representation corresponding to $\lambda$, and $\chi^\lambda$ for its character. Of particular interest for us are the (dual) irreducible reflection representation $V^*\is S^{(a-1,1)}$, the permutation representation $\c^a\is S^{(a-1,1)} \oplus S^{(a)}$, and also its higher exterior powers $\wedge^k \c^a \is S^{(a-k, 1^k)} \oplus S^{(a-k+1, 1^{k-1})}$.

To state the Molchanov result, the following notation will be useful. Given a partition $\l=(\l_1,\dots\l_r)$, write $(i,j)\in\l$ to mean that $i$ and $j$ are positive integers with $i\leq r$ and $j\leq \lambda_i$. This notation is justified by thinking of $\lambda$ as its Ferrers diagram in French notation:

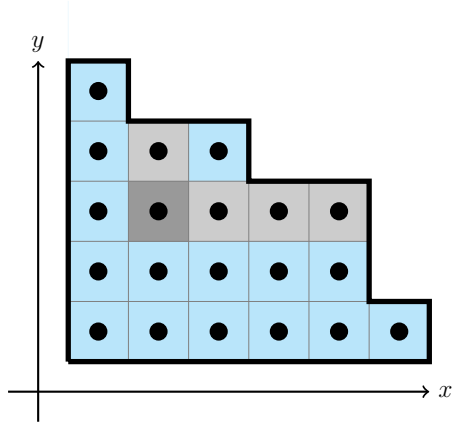
\vspace{-0.1in}
\begin{center}
\begin{tikzpicture}[thick,scale=0.8, every node/.style={scale=0.8},baseline={([yshift=-.5ex]current bounding box.center)}]

\begin{scope}[shift={(-.5,-.5)}]
  \fill[cyan!25]
  (1,1)
  \foreach \iter [count=\y] in {6,5,5,3,1,0} {
      \ifnum\y > 1
        -- ++(0,1)
      \fi
      -- (\iter+1,\y)    
  } 
  -- ++(0,1)
  |- (1,1);
\end{scope}


\begin{scope}[shift={(-.5,-.5)}]
  \fill[gray!40]
  (2,3)
  \foreach \iter [count=\y] in {4,1,0} {
      \ifnum\y > 1
        -- ++(0,1)
      \fi
      -- (\iter+2,\y+2)    
  } 
  -- ++(0,1)
  |- (2,3);

  \fill[gray!80]
  (2,3) -- (2,4) -- (3,4) -- (3,3) -- (2,3);

  \draw[help lines] (1,2)--(7,2);
  \draw[help lines] (1,3)--(6,3);
  \draw[help lines] (1,4)--(6,4);
  \draw[help lines] (1,5)--(3,5);
  \draw[help lines] (2,1)--(2,5);
  \draw[help lines] (3,1)--(3,5);
  \draw[help lines] (4,1)--(4,4);
  \draw[help lines] (5,1)--(5,4);
  \draw[help lines] (6,1)--(6,2);

\end{scope}

  \draw[fill] (1,1) circle [radius=0.1333];
  \draw[fill] (2,1) circle [radius=0.1333];
  \draw[fill] (3,1) circle [radius=0.1333];
  \draw[fill] (4,1) circle [radius=0.1333];
  \draw[fill] (5,1) circle [radius=0.1333];
  \draw[fill] (6,1) circle [radius=0.1333];
  \draw[fill] (1,2) circle [radius=0.1333];
  \draw[fill] (2,2) circle [radius=0.1333];
  \draw[fill] (3,2) circle [radius=0.1333];
  \draw[fill] (4,2) circle [radius=0.1333];
  \draw[fill] (5,2) circle [radius=0.1333];
  \draw[fill] (1,3) circle [radius=0.1333];
  \draw[fill] (2,3) circle [radius=0.1333];
  \draw[fill] (3,3) circle [radius=0.1333];
  \draw[fill] (4,3) circle [radius=0.1333];
  \draw[fill] (5,3) circle [radius=0.1333];
  \draw[fill] (1,4) circle [radius=0.1333];
  \draw[fill] (2,4) circle [radius=0.1333];
  \draw[fill] (3,4) circle [radius=0.1333];
  \draw[fill] (1,5) circle [radius=0.1333];

\begin{scope}[shift={(-.5,-.5)}]
  \draw[line width=2pt]
  (1,1)
  \foreach \iter [count=\y] in {6,5,5,3,1}{
      \ifnum\y > 1
        -- ++(0,1)
      \fi
      -- (\iter+1,\y)    
  } 
  -- ++(0,1)
  -- (1,6) -- (1,1);

\end{scope}

\draw[->, thick] (-.5,0)--(6.5,0) node[right]{$x$};
\draw[->, thick] (0,-.5)--(0,5.5) node[above]{$y$};

  \end{tikzpicture}
\end{center}

When $(i,j)\in\l$, the \textbf{hook-length} $h(i,j)$ is defined as $(\l_i-j)+(\l'_j-i)+1$, where $\lambda'$ is the conjugate partition given by $(i,j)\in\l \Leftrightarrow (j,i)\in \l'$. The diagram above illustrates this definition, in particular that for $\l=(6,5,5,3,1)$ we have $h(2,3)=5$.

\begin{thm}[Molchanov \cite{Mol1992}]
\label{kirkman}
Let $S^iV^*$ be the $i^\text{th}$ graded component of $SV^*$. Then
$$ \sum_{i,j\geq 0} \dim\Hom_{\Bbb C [S_a]}\!\big(S^\lambda, S^iV^*\otimes \wedge^j V\big)\cdot q^i t^j ~=~ \frac{1-q}{1+t}\prod_{(i,j)\in\lambda} \frac{q^{i-1}+tq^{j-1}}{1-q^{h(i,j)}}. $$

\end{thm}

\begin{proof}[Proof (of Proposition \ref{schroderformula})]

From Theorem \ref{hsop}, we obtain the following identity on graded $\c[S_a]$-characters:
$$ \chi_{L_{b/a}(1)} = \sum_{j=0}^{a-1} (-1)^j\chi_{SV^*\otimes \wedge^{j} U}. $$
From this we conclude that a similar identity holds for each $\lambda$-isotypic component:
$$ \mathcal H \Big( \!\Hom_{\c[S_a]}\big(S^\lambda, L_{b/a}(1)\big); ~q\Big) ~=~ \sum_{j=0}^{a-1} (-1)^j \, \mathcal H \Big( \!\Hom_{\c[S_a]}\big(S^\lambda, SV^*\otimes \wedge^j U\big); ~q\Big). $$

Recall that using Dunkl's construction (again, recorded as Theorem \ref{hsopsexist}), $U\is V^*$ as ungraded $S_a$-representations. Therefore, the right-hand side is almost in the same form as Theorem \ref{kirkman}, except that we have lost $t^j$, picked up a factor of $(-1)^j$, and in the exterior powers, the degree-$1$ elements of $V^*$ are now the degree-$b$ elements of $U$. These differences are not so severe; they simply amount to evaluating Theorem \ref{kirkman} at $t=-q^b$:
\begin{equation}
\label{cherednik-hilbert-equality}
 \mathcal H \Big( \!\Hom_{\c[S_a]}\big(S^\lambda, L_{b/a}(1)\big); ~q\Big) ~=~ \frac{1-q}{1-q^b}\prod_{(i,j)\in\lambda} \frac{q^{i-1}-q^{j-1+b}}{1-q^{h(i,j)}}\,.
\end{equation}

To complete the calculation of the rational $q$-Schr\"oder numbers, we note that the character of $\wedge^k \c^a \is \wedge^k V^* \oplus \wedge^{k-1} V^*$ is $\chi^{({a-k},1^{k})}+\chi^{({a-k+1},1^{k-1})}$ for all $k>0$, and also for $k=0$ under the reasonable convention that $\chi^{(a+1,1^{-1})}:=0$. By applying the above formula twice with these hook shapes and collecting common terms, we obtain the product formula for $C(k, a-k, b-k;q)$.
\end{proof}

\subsection{The $\mathfrak{sl}_2(\c)$ Action on $L_{b/a}(1)$}

Given an algebra $A$ equipped with an action of a group $G$, the \textbf{semidirect product} $A\rtimes G$ is the algebra which as a vector space is $A\otimes \c[G]$, and whose product structure given by $ (a\otimes g)\cdot (b\otimes h) = ag(b) \otimes gh$.

Let $y_1,\dots, y_{a-1}$ and $x_1,\dots, x_{a-1}$ be respectively a basis for $V$ and its dual basis. The \textbf{rational Cherednik algebra} is $H_{b/a} = (S(V\oplus V^*)\rtimes S_a)/I$, where $I$ is the ideal generated by the following relations ($1\leq i,j\leq a-1$):
\begin{align*}
x_ix_j &=x_jx_i \\
y_iy_j & =y_jy_i \\
x_iy_j &= y_jx_i \qquad~\text{  for all } i\neq j \\
x_iy_i-y_ix_i &= 1 - \ds \frac{b}{a}\sum_{\substack{k=1 \\ k\neq i}}^{a} (ik).
\end{align*}

This algebra can be given a grading via $\deg(w)=0$ for all $w\in S_n$, and for the variables, $\deg(x_i)=1$ and $\deg(y_i)=-1$. 

Moreover, we may regard $SV^*$ as an $H_{b/a}$-module in the following way: Let $\c$ be the trivial $S_a$-representation, and extend it to an $(SV\rtimes S_a)$-module by letting each $y_i$ act as zero. Then there is an isomorphism of $(SV^*\rtimes S_a)$- and hence $\c[S_a]$-modules: $SV^* \is H_{b/a}\otimes_{SV\rtimes S_a} \c$. (Note that this does not mean that left-multiplication by $y_i$ acts on $SV^*$ as zero, because before $y_i$ acts as zero on the right tensor factor $\c$, it must first be commuted past each $x_j$ in the left tensor factor $H_{b/a}$, which may introduce $\c[S_a]$ terms. We have avoided discussing the explicit action of $SV$ on $SV^*$ because it is somewhat involved; but see \cite[\S 2.5 and \S 3.1]{EtiMa2010} for details.) 

With this module structure in hand, we may finally record the existence result, primarily due to Dunkl:

\begin{thm}[Dunkl {\cite[\S 5]{Dun1998}, {\cite[\S 6]{Dun2005}}} \& Dunkl-Opdam {\cite[Prop 2.34]{DunOpd2003}}]
\label{hsopsexist}
For any $a< b$ where $b$ is not a multiple of $a$, there is a space $U$ of degree-$b$ polynomials in $SV^*$ such that $U\is V^*$ as $S_a$-representations, and $y_i\cdot U = 0$ for all $1\leq i\leq a-1$.

In particular, the ideal $\<U\>\subset SV^*$ is an $H_{b/a}$-submodule. If moreover $\gcd(a,b)=1$, then the quotient $SV^*/\!\<U\>$ is finite-dimensional.
\end{thm}

Furthermore, we are now ready to prove Theorem \ref{parityunimodal}:
\parityunimodal*

\begin{proof}
This argument is loosely based on Haiman \cite[\S 7]{Hai1994}, which uses simpler tools to obtain the result in the $b=a+1$ case.

As in \cite[\S 3]{BerEtiGin2003-1}, there is an action of $\mathfrak{sl}_2(\c) = \c\{e,f,h\}$ on $H_{b/a}$ given by left multiplication of certain elements:
\begin{align*}
e &= -\frac12\sum_{i=1}^{a-1} x_i^2 \\
f &= \frac12\sum_{i=1}^{a-1} y_i^2 \\
h &= \sum_{i=1}^{a-1} x_iy_i + (a-1)\left(\frac12-\sum_{i=1}^{a-1}(i,i+1)\right).
\end{align*}

Hence, the element $e$ acts on $U\ss SV^*$ by the natural left-multiplication, and $h$ acts as left-multiplication by $(a-1)/2 + (a-1)\sum_{i=1}^{a-1} (i, i+1) \in \c[S_a]$; in particular, since $U$ is an $S_a$-subrepresentation of $SV^*$ we have $h\cdot\theta\in U$ as well. Finally, recalling from Theorem \ref{hsopsexist} that $U\ss SV^*$ consists exclusively of vectors $u$ such that $y\cdot u=0$, we have that $f$ acts on $U$ as zero.

Therefore the $\mathfrak{sl}_2(\c)$ action preserves $\<U\>$, and so $L_{b/a}(1)=SV^*/\!\<U\>$ admits an $\mathfrak{sl}_2(\c)$ action. Moreover, this action commutes with the action of $S_a$, because $e$ and $f$ are clearly invariant under permuting indices (and thus, so is $h=ef-fe$).

Any finite-dimensional $\mathfrak{sl}_2(\c)$-module $V$ has a \textbf{formal character} $\text{ch}(V) = \sum \dim(V_\ell)q^\ell$, where $V_\ell\ss V$ is the space of all elements having weight $\ell$. By typical Lie theory arguments (see, e.g. \cite[Theorem 15]{Sta1989}), a formal character is a Laurent polynomial that is symmetric and parity-unimodal about $q^0$.

The signficance for our situation is that the grading on $H_{b/a}$ descends to a grading on $L_{b/a}(1)$, and since $h$ preserves the grading on $H_{b/a}$ it also does so on $L_{b/a}(1)$. It follows that for any graded $V\ss L_{b/a}(1)$, weight differs from degree only by a constant shift. We conclude that $\text{ch}(V)$ is the Hilbert series of $V$ up to a factor of some $q^c$. Since both $\mathfrak{sl}_2(\c)$ and $\c[S_a]$ are semisimple, we may write
$$L_{b/a}(1) = \bigoplus_{\l\vdash a} \left(\bigoplus_{\ell\geq 0} m_{\ell,\l}V^m\right)\otimes S^\l.$$

Hence the space of intertwiners of $S^\l$ with $L_{b/a}(1)$ has (shifted) Hilbert function $P_\l(q)$, where the $P_\l$ are each Laurent polynomials, symmetric and parity-unimodal about $q^0$. In particular, $q^{-c}C_{a,b}^k(q) = P_{({a-k},1^{k})}(q)+P_{(a-k-1,1^{k+1})}(q)$ is symmetric and parity-unimodal about $q^0$, which is equivalent to the desired statement.
\end{proof}

\section{Group-Theoretic Formulation}
\label{gpthy}

Turning our attention to cyclic sieving, we begin by reviewing a cyclic sieving phenomenon that specializes a result from \cite{ReiStaWhi2004}. To avoid some trivialities, we assume for the remainder of the paper that $n\geq 3$.

Given any subgroup $H$ of $S_n$, let $X$ be the coset space $X=S_n/H$, and $C=\langle c \rangle$ be the cyclic subgroup of $S_n$ generated by the $n$-cycle $c=(1,2,\ldots,n)$. Recall that $S_n$, and hence $H$, acts on the graded ring of $n$-variable polynomials $\c[x]=\c[x_1,\dots, x_n]$ by permuting indices. (Note that, unlike in the previous section this polynomial ring has $n$ variables, agreeing with the index of $S_n$.) Denote the fixed space of this $S_n$-action by let $\c[x]^{S_n}$, and similarly for $\c[x]^H$. 
Then \cite[Theorem 8.2]{ReiStaWhi2004} implies that the triple $(X,X(q),C)$ exhibits the cyclic sieving phenomenon, where
\begin{equation}
\label{coset-polynomial}
X(q)= \frac{ \mathcal H(\c[x]^H ,q) }{ \mathcal H(\c[x]^{S_n},q) }.
\end{equation}

\begin{remark}
Somewhat different notation is used in \cite{ReiStaWhi2004}: their $X(q)$ is defined as $\mathcal H(A(S_n)^{H},q)$, where $A(S_n)$ is the coinvariant algebra $\c[x]/\<f\in \c[x]^{S_n}: f(0)=0\>$. The statement that $\mathcal H(A(S_n)^H;q)$ is the same as $\mathcal H(\c[x]^H;q)/\mathcal H(\c[x]^{S_n}; q)$ is a standard fact from invariant theory; see for instance \cite[Corollary 1.2.2]{BroReiSmiWeb2011}.
\end{remark}

We write $\Conj_G(\gamma)$ to denote the set of elements in a group $G$ conjugate to $\gamma$. For $G=S_n$ we can describe these concretely. Recall that by counting the lengths of cycles in any $\gamma\in S_n$, and placing them in non-decreasing order, we obtain a partition $\mu$; we say that $\mu=\cyc(\gamma)$ is the \textbf{cycle type} of $\gamma$. Then elements are conjugate in $S_n$ if and only if they have the same cycle type. Therefore, as an abuse of notation, we may write $\Conj_{S_n}(\mu)$ instead of $\Conj_{S_n}(\gamma)$.

Generally, we say that $H$ \textbf{avoids} $\mu$ if $H\cap\Conj_{S_n}(\mu)=\varnothing$. We will be interested in $X=S_n/H$ as a set on which $C$ acts by left-multiplication, particularly for those $H$ such that the action is free. Note that the freeness of this action is equivalent to the condition that no nontrivial power of $c$ is $S_n$-conjugate to an element of $H$, and hence to the statement that $H$ avoids $(d^{\frac{n}{d}})$ for any divisor $d>1$ of $n$. In this case, we aim to set up an additional cyclic sieving triple. We begin with the polynomial:

\begin{prop}
\label{freeactpoly}
Let $C$ be a cyclic group acting freely on a set $X$, and $X(q)\in \z[q]$. Then $(X,X(q),C)$ exhibits the cyclic sieving phenomenon if and only if $Y(q) = \frac{1}{[n]_q}X(q)$ is a polynomial in $\z[q]$.
\end{prop}

\begin{proof}
Let $\zeta$ be a primitive $n^\text{th}$ root of unity. Then both conditions are equivalent to the fact that that $X(\zeta^i)=0$ for any $1\leq i\leq n-1$, because $[n]_q=\frac{1-q^n}{1-q}=\prod_{i=1}^{n-1}(q-\zeta^i)$.
\end{proof}

Moreover, notice that elements $\tau$ of the normalizer $N_{S_n}(C)$ can act on $Y=C\backslash S_n/H$, the collection of double-cosets $CgH$, via the rule
\begin{equation}
\label{normalizer-element-action}
\tau \cdot CgH =\tau CgH = C\tau g H.
\end{equation}

\begin{exA}
\label{necklaceinvolution}
For instance, let $\tau_0$ be the permutation that fixes $n$ and otherwise sends $i$ to $n-i$, for any $1\leq i\leq n$. Note that $\tau_0$ also fixes exactly one other vertex, namely $\frac{n}{2}+1$, when $n$ is even. In that case that $H=S_{\alpha_1}\times\cdots\times S_{\alpha_r}$ this has a natural geometric interpretation on words, which descends to (the unique) reflection on necklaces.

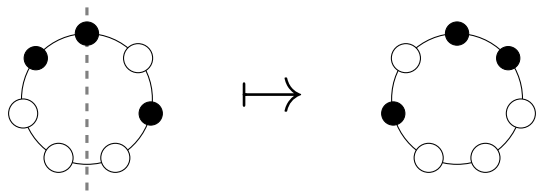
\begin{center}
\begin{tikzpicture}[scale=1]
\pgfmathsetlengthmacro{\dot}{5.5pt};
\pgfmathsetlengthmacro{\layerheight}{48pt};
\pgfmathsetlengthmacro{\heightoffset}{\layerheight / 2};

\pgfmathsetmacro{\b}{3};
\pgfmathsetmacro{\a}{4};
\pgfmathsetlengthmacro{\sqsize}{8pt};

\pgfmathsetlengthmacro{\neckR}{\dot*4.5};
\pgfmathsetmacro{\k}{0};
\pgfmathsetmacro{\beads}{\a+\b-\k};
\pgfmathsetmacro{\beadSep}{360/\beads};
\pgfmathsetmacro{\beadOff}{90};
\pgfmathsetlengthmacro{\beadR}{\dot};

\coordinate (top1) at (0,\heightoffset+3*\layerheight);
\coordinate (top2) at (0,\heightoffset+2*\layerheight);
\coordinate (top3) at (0,\heightoffset+1*\layerheight);
\coordinate (TL) at (0pt,\heightoffset);
\coordinate (TR) at (80pt,\heightoffset);
\coordinate (BL) at (0pt,-\heightoffset);
\coordinate (BR) at (80pt,-\heightoffset);
\coordinate (bot3) at (0,-\heightoffset-1*\layerheight);
\coordinate (bot2) at (0,-\heightoffset-2*\layerheight);
\coordinate (bot1) at (0,-\heightoffset-3*\layerheight);

\begin{scope}[shift={(70pt,-\heightoffset)}]
	\coordinate (center) at (0, 0);
	\draw [black,] (center) circle [radius=\neckR];
	\draw [black, fill=black] (center)+(\beadOff-0*\beadSep:\neckR) circle [radius=\beadR*0.8];
	\draw [black, fill=black] (center)+(\beadOff-1*\beadSep:\neckR) circle [radius=\beadR*0.8];
	\draw [black, fill=black] (center)+(\beadOff-5*\beadSep:\neckR) circle [radius=\beadR*0.8];
	\draw [black, fill=white, thin] (center)+(\beadOff-2*\beadSep:\neckR) circle [radius=\beadR];
	\draw [black, fill=white, thin] (center)+(\beadOff-3*\beadSep:\neckR) circle [radius=\beadR];
	\draw [black, fill=white, thin] (center)+(\beadOff-4*\beadSep:\neckR) circle [radius=\beadR];
	\draw [black, fill=white, thin] (center)+(\beadOff-6*\beadSep:\neckR) circle [radius=\beadR];
\end{scope}

\node at (0,-\heightoffset) {\color{black} \Huge $\mapsto$};

\begin{scope}[shift={(-70pt,-\heightoffset)}]
	\coordinate (center) at (0, 0);
	\draw [black,] (center) circle [radius=\neckR];
	\draw [dashed, gray, very thick] (center)+(0,1.4*\neckR) -- (center);
	\draw [dashed, gray, very thick] (center)+(0,-1.4*\neckR) -- (center);
	\draw [black, fill=black] (center)+(\beadOff+0*\beadSep:\neckR) circle [radius=\beadR*0.8];
	\draw [black, fill=black] (center)+(\beadOff+1*\beadSep:\neckR) circle [radius=\beadR*0.8];
	\draw [black, fill=black] (center)+(\beadOff+5*\beadSep:\neckR) circle [radius=\beadR*0.8];
	\draw [black, fill=white, thin] (center)+(\beadOff+2*\beadSep:\neckR) circle [radius=\beadR];
	\draw [black, fill=white, thin] (center)+(\beadOff+3*\beadSep:\neckR) circle [radius=\beadR];
	\draw [black, fill=white, thin] (center)+(\beadOff+4*\beadSep:\neckR) circle [radius=\beadR];
	\draw [black, fill=white, thin] (center)+(\beadOff+6*\beadSep:\neckR) circle [radius=\beadR];
\end{scope}

\end{tikzpicture}
\end{center}

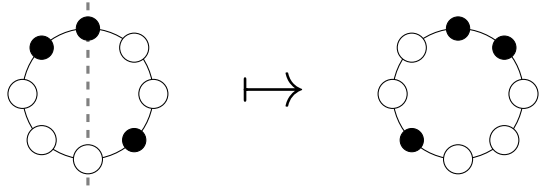
\begin{center}
\begin{tikzpicture}[scale=1]
\pgfmathsetlengthmacro{\dot}{5.5pt};
\pgfmathsetlengthmacro{\layerheight}{48pt};
\pgfmathsetlengthmacro{\heightoffset}{\layerheight / 2};

\pgfmathsetmacro{\b}{3};
\pgfmathsetmacro{\a}{5};
\pgfmathsetlengthmacro{\sqsize}{8pt};

\pgfmathsetlengthmacro{\neckR}{\dot*4.5};
\pgfmathsetmacro{\k}{0};
\pgfmathsetmacro{\beads}{\a+\b-\k};
\pgfmathsetmacro{\beadSep}{360/\beads};
\pgfmathsetmacro{\beadOff}{90};
\pgfmathsetlengthmacro{\beadR}{\dot};

\coordinate (top1) at (0,\heightoffset+3*\layerheight);
\coordinate (top2) at (0,\heightoffset+2*\layerheight);
\coordinate (top3) at (0,\heightoffset+1*\layerheight);
\coordinate (TL) at (0pt,\heightoffset);
\coordinate (TR) at (80pt,\heightoffset);
\coordinate (BL) at (0pt,-\heightoffset);
\coordinate (BR) at (80pt,-\heightoffset);
\coordinate (bot3) at (0,-\heightoffset-1*\layerheight);
\coordinate (bot2) at (0,-\heightoffset-2*\layerheight);
\coordinate (bot1) at (0,-\heightoffset-3*\layerheight);

\begin{scope}[shift={(70pt,-\heightoffset)}]
	\coordinate (center) at (0, 0);
	\draw [black,] (center) circle [radius=\neckR];
	\draw [black, fill=black] (center)+(\beadOff-0*\beadSep:\neckR) circle [radius=\beadR*0.8];
	\draw [black, fill=black] (center)+(\beadOff-1*\beadSep:\neckR) circle [radius=\beadR*0.8];
	\draw [black, fill=black] (center)+(\beadOff-5*\beadSep:\neckR) circle [radius=\beadR*0.8];
	\draw [black, fill=white, thin] (center)+(\beadOff-2*\beadSep:\neckR) circle [radius=\beadR];
	\draw [black, fill=white, thin] (center)+(\beadOff-3*\beadSep:\neckR) circle [radius=\beadR];
	\draw [black, fill=white, thin] (center)+(\beadOff-4*\beadSep:\neckR) circle [radius=\beadR];
	\draw [black, fill=white, thin] (center)+(\beadOff-6*\beadSep:\neckR) circle [radius=\beadR];
	\draw [black, fill=white, thin] (center)+(\beadOff-7*\beadSep:\neckR) circle [radius=\beadR];
\end{scope}

\node at (0,-\heightoffset) {\color{black} \Huge $\mapsto$};

\begin{scope}[shift={(-70pt,-\heightoffset)}]
	\coordinate (center) at (0, 0);
	\draw [black,] (center) circle [radius=\neckR];
	\draw [dashed, gray, very thick] (center)+(0,1.4*\neckR) -- (center);
	\draw [dashed, gray, very thick] (center)+(0,-1.4*\neckR) -- (center);
	\draw [black, fill=black] (center)+(\beadOff+0*\beadSep:\neckR) circle [radius=\beadR*0.8];
	\draw [black, fill=black] (center)+(\beadOff+1*\beadSep:\neckR) circle [radius=\beadR*0.8];
	\draw [black, fill=black] (center)+(\beadOff+5*\beadSep:\neckR) circle [radius=\beadR*0.8];
	\draw [black, fill=white, thin] (center)+(\beadOff+2*\beadSep:\neckR) circle [radius=\beadR];
	\draw [black, fill=white, thin] (center)+(\beadOff+3*\beadSep:\neckR) circle [radius=\beadR];
	\draw [black, fill=white, thin] (center)+(\beadOff+4*\beadSep:\neckR) circle [radius=\beadR];
	\draw [black, fill=white, thin] (center)+(\beadOff+6*\beadSep:\neckR) circle [radius=\beadR];
	\draw [black, fill=white, thin] (center)+(\beadOff+7*\beadSep:\neckR) circle [radius=\beadR];
\end{scope}

\end{tikzpicture}
\end{center}
\end{exA}

We now write a generalization of Theorem \ref{braceletcount} that allows some flexibility with both $\tau$ and $H$:

\begin{thm}
\label{technicalcsp}
Fix an element $\tau\in N_{S_n}(C)$ whose cycle type is either $(m^{\frac{n-1}{m}},1)$ or $(m^{\frac{n-2}{m}},1,1)$,
for some integer $m$. Suppose that $H\leq S_n$ is a subgroup that avoids the following cycle types:
\begin{itemize}
\item $(\ell^\frac{n}{\ell})$ for any divisor $\ell>1$ of $n$
\item $(4, 2^{\frac{n-4}{2}})$ if $m$ is even
\item $\big( \ell^{\frac{n-2}{\ell}}, 2 \big)$ for any divisor $\ell>1$ of $m$
\item $\big( (2\ell)^{\frac{n-2}{2\ell}}, 2 \big)$ for any divisor $\ell>1$ of $m$, if $m$ is odd
\end{itemize} 
Finally, let $\<\tau\>$ act on $Y=C\backslash S_n/H$ via the rule (\ref{normalizer-element-action}) and $Y(q)=\frac{1}{[n]_q}X(q)$, where X(q) is defined by (\ref{coset-polynomial}). Then $(Y, Y(q), \<\tau\>)$ exhibits the cyclic sieving phenomenon.
\end{thm}

The technicalities here are unfortunate, but the restrictions on $H$, at least, capture genuine difficulties. For instance, the desired sieving fails for $H=\<(1234)(5678)(90)\>\leq S_{10}$ and $\tau=(1)(2408)(3795)(6)$, even though $\tau\in N_{S_{10}}(C)$ has cycle type $(4^{\frac{10-2}{2}},1,1)$.

On the other hand, it may be possible to allow a broader class of $\tau$ if we appropriately restrict $H$, but the restrictions on $\tau$ given here are needed for our argument. The precise role they play is explicated at the end of Section \ref{yzetacalc}, where in particular it is clear that these are the only cycle types that can reasonably be expected to yield such a cyclic sieving result whenever $n\equiv 1,2\bmod m$. However, when $n\not\equiv 1,2\bmod m$ the situation appears much more delicate, and we do not have a general conjecture.

Nevertheless, this theorem is already permissive enough to resolve Theorem \ref{braceletcount}, as follows.

\braceletcount*

\begin{proof}[Proof (of Theorem \ref{braceletcount})] 
It is a standard fact of invariant theory that if $H=S_{\a_1}\times\cdots\times S_{\a_r}$, then
$$ \mathcal H(\c[x]^{S_n},q) = \prod_{i=1}^n \frac{1}{1-q^i}
\qquad\text{and}\qquad
\mathcal H(\c[x]^{H},q) = \prod_{j=1}^r \prod_{i=1}^{\a_j} \frac{1}{1-q^i}. $$
From this we deduce that
$$ X(q)  =  \frac{ \mathcal H(\c[x]^H ,q) }{ \mathcal H(\c[x]^{S_n},q) }  =  \pvec{n}{\a}_q.$$

Notice that as $S_n$-sets, $X=S_n/H$ is equivalent to the set of words having exactly $\a_i$ occurences of the letter $i$, and so $C$ acts freely on $X$ if and only if $\gcd(\a_1,\dots, \a_r)=1$. In this case, the associated $Y(q)$ is $C(\a;q)$. In Example A, we saw that $\tau_0$ acts by on $Y$ by reflection, and that its cycle type is $(2^{\frac{n-1}{2}},1)$ for odd $n$ and $(2^{\frac{n-2}{2}},1,1)$ for even $n$.

Moreover, for this choice of $\tau$ (for which $m=2$), we observe that:
\begin{itemize}
\item As discussed above, the fact that $C$ acts freely on $X$ is equivalent to $H$ avoiding $(\ell^\frac{n}{\ell})$ for any divisor $\ell>1$ of $n$.
\item $H$ cannot contain elements with cycle type $(2^{\frac{n-4}{2}},4)$, because otherwise every $\a_i$ would have to be even, but $\gcd(\a)=1$.
\item The only divisor of $2$ aside from $1$ is $\ell=2$ itself, for which $\big( \ell^{\frac{n-2}{\ell}}, 2 \big) = (2^{\frac{n}{2}})$. Again $H$ avoids this cycle type by the freeness of $C$ on $X$.
\end{itemize}

Therefore, $\tau_0$ and $H$ satisfy the conditions of Theorem \ref{technicalcsp}, and thus we conclude that the triple $(Y,Y(q), \tau_0)$ exhibits Stembridge's $q=-1$ phenomenon, as desired.
\end{proof}

Before beginning the proof of Theorem \ref{technicalcsp}, we wish to make two more remarks.

First, it is clear that the latter three cycle conditions apply only when $n\equiv 2 \bmod m$. It is tempting to think that the only problem with extending to $n\equiv 3\bmod m$ is an unwieldy proliferation of cycle type restrictions. This may indeed be the case, but we reiterate that our argument breaks in a more substantive way.

Second, in Section \ref{fixycalc} we recall some facts from elementary number theory that provide some insight into which $\tau\in S_n$ have the cycle types required by Theorem \ref{technicalcsp}. In particular, this reveals a fairly general setting in which all of the technicalities simplify. When $n$ is an odd prime, it happens that every $\tau\in N_{S_n}(C)$ is either in $C$ itself, or has cycle type $(m^{\frac{n-1}{m}},1)$ for some $m$. Moreover, as described above, we only need the first cycle type restriction. Therefore, we obtain the following pleasing corollary:

\begin{cor}
\label{primecsp}
Let $p$ be an odd prime. Fix an element $\tau\in N_{S_p}(C)\minus C$, and a subgroup $H\leq S_p$ for which $C$ acts freely on $S_p/H$. Additionally, let $\<\tau\>$ act on $Y=C\backslash S_p/H$ via the rule (\ref{normalizer-element-action}) and let $Y(q)=\frac{1}{[n]_q}X(q)$, where $X(q)$ is defined by (\ref{coset-polynomial}). Then the triple $(Y, Y(q), \<\tau\>)$ exhibits the cyclic sieving phenomenon.
\end{cor}

\section{Proof of Theorem \ref{technicalcsp}}

We begin by fixing some notation for the remainder of the section. 
For any group $G$ acting on some set $A$, and any $g\in G$, write $\Fix_A(g)$ to denote the set of $g$-fixpoints: $\{a\in A: g\cdot a=a\}$. Any two $G$-conjugate elements have the same number of fixpoints in $A$. So, in particular, for a partition $\mu$ of $n$, we abuse notation and write $\left|\Fix_{A}(\mu)\right|$ to mean the number of points in $A$ that are fixed by any permutation with cycle type $\mu$.

In the following two subsections we will complete the bulk of a single root-of-unity calculation, and then we will bundle them together with some concluding details. For the intermediate results, the following definition is useful:

\begin{defn}
Suppose that $\tau'\in N_{S_n}(C)$ has cycle type either $\big(m^{\frac{n-1}{m}},1\big)$ or $\big(m^{\frac{n-2}{m}},1,1\big)$,
for some integer $m$. Write $k=\left\lfloor\frac{n-1}{m}\right\rfloor$ for the number of $m$-cycles that $\tau'$ has. Moreover, suppose that $H\leq S_n$ is a subgroup such that $H$ avoids the following cycle types:
\begin{itemize}
\item $(\ell^\frac{n}{\ell})$ for any divisor $\ell>1$ of $n$
\item $(m^{k}, 2)$, and
\item $( (2m)^{\frac{k}{2}},2 )$ if $m$ is odd.
\end{itemize} 
In this case, we say that the pair $(\tau',H)$ is \textbf{\goodpair}.
\end{defn}

Part of the bundling process is the observation that the conditions of Theorem \ref{technicalcsp} on $\tau$ and $H$ are equivalent to the statement that $(\tau^b,H)$ is \goodpair\!\, for every integer $1\leq b\leq m-1$, and also $H$ avoids $(4,2^{\frac{n-4}{2}})$. We will see that the latter cycle restriction arises from a different consideration than $C$-admissibilitydoes.

\subsection{Evaluating $Y(\zeta)$}
\label{yzetacalc}

The following lemma gives an explicit connection between $Y(\zeta)$ and various fixpoints in $X$. Notice that we only use the freeness condition on $H$, and not the other cycle type restrictions.

\begin{prop}
\label{moliencalc}
Fix a subgroup $H\leq S_n$ such that $C$ acts freely on $X=S_n/H$. Fix an integer $m\geq 2$, and then define $\zeta$ to be a primitive $m^\text{th}$ root of unity, and $k=\left\lfloor\frac{n-1}{m}\right\rfloor$. Moreover, for any partition $\l$ of $n-km$, write $c_i$ to denote the number of parts in $\l$ with size $i$. Then, defining
$$ X(q) = \frac{ \mathcal H(\c[x]^H ,q) }{ \mathcal H(\c[x]^{S_n},q) } $$
and $Y(q)=\frac{1}{[n]_q}X(q)$, we have the following:

\begin{enumerate}[(a)]
\item If $n\not\equiv 0 \text{ mod } m$, then 
$$ Y(\zeta) = (1-\zeta)\ds\prod_{i=1}^{n-1-km} \!\! (1-\zeta^i) \left[ \sum_{\l\,\vdash\, n-km}\frac{\left|\Fix_X(m^k,\l)\right|}{~\ds\prod_{i\geq 1} (i(1-\zeta^i))^{c_i}c_i!~}\right] .$$

\item If $n\equiv 0 \text{ mod } m$, then 
$$ Y(\zeta) = (1-\zeta) \left[ \frac{mk}{4}\left|\Fix_X(2m, m^{k-1})\right| ~+~ \sum_{\substack{\l\,\vdash\, m \\ \l_1\neq m}}\frac{\left|\Fix_X(m^k,\l)\right|}{~\ds\prod_{i\geq 1} (i(1-\zeta^i))^{c_i}c_i!~}  \right].$$
\end{enumerate}
\end{prop}

\begin{proof}
Observe that
\begin{align*}
X(q) &= \frac{\mathcal H(\c[x]^H;q)}{\mathcal H(\c[x]^{S_n};q)} \\
 &= \mathcal H(\c[x]^H;q)\prod_{i=1}^n (1-q^i) \\
 &= \mathcal H(\c[x]^H;q) \cdot (1-q)^n[n]!_q.
\end{align*}

Thus, $Y(q) = (1-q)^n [n-1]!_q \cdot \mathcal H(\c[x]^H;q)$. We can explicitly calculate the Hilbert series of the $H$-invariants using Molien's formula \cite{Mol1897}:
$$ \mathcal H(\c[x]^H;q) = \frac{1}{|H|}\sum_{h\in H} \frac{1}{1-\det(I-qh)} = \frac{1}{|H|}\sum_{h\in H}\prod_{\text{cycles } z \text{ of } h} \frac{1}{1-q^{|z|}}, $$
where $I$ is the identity map on $\c^n$, $h$ is the permutation matrix representing its action on (linear combinations of) the variables, and $|z|$ is the length of the cycle $z$.

Putting the Hilbert series aside momentarily, notice that
\begin{align*}
(1-q)^n[n-1]!_q & = (1-q) \prod_{i=1}^{n-1} (1-q^i) \\
 &= (1-q)\prod_{j=1}^k \left((1-q^{jm})\prod_{r=1}^{m-1} (1-q^{jm+r})\right)\prod_{i=km+1}^{n-1} (1-q^i).
\end{align*}
For each $h\in H$, define the auxiliary quantity
$$F_h(\zeta) := \lim_{q\to\zeta} ~ \prod_{j=1}^k (1-q^{jm}) \cdot \prod_{\text{cycles } z \text{ of } h} \frac{1}{1-q^{|z|}}, $$
so that 
$$Y(\zeta) = \frac{1}{|H|} (1-\zeta) \cdot \left(\prod_{r=1}^{m-1} (1-\zeta^{i'})\right)^k \cdot \prod_{i=km+1}^{n-1} (1-\zeta^i) \cdot  \sum_{h\in H} F_h(\zeta). $$

The second of the four factors in the above expression has a simple evaluation. Because $x=\zeta^{i'}$ is a root of $[m]_x$ for all integers $1\leq i'\leq m-1$, we conclude that $\prod_{i'=1}^m(x-\zeta^{i'})=[m]_x$, and hence 
$$Y(\zeta) = \frac{1}{|H|} (1-\zeta) \cdot m^k \cdot \prod_{i=km+1}^{n-1} (1-\zeta^i) \cdot  \sum_{h\in H} F_h(\zeta). $$

It remains to compute $F_h(\zeta)$. The first factor of $F_h(\zeta)$ has a zero of multiplicity $k$ at $q=\zeta$, and so $F_h(\zeta)=0$ unless the second factor has a pole of multiplicity at least $k$ at $q=\zeta$. We can see that this occurs precisely when $h$ has at least $k$ cycles whose lengths divide $m$. By definition of $k$, the element $h$ can never have more than $k+1$ cycles whose lengths divide $m$, because $(k+1)m\geq n$. In fact, $h$ has at most $k$ such cycles: equality occurs if and only if $m|n$, but then $\cyc(h)=(m^k)$, which contradicts that $(\tau, H)$ is \goodpair. This, in turn, means that either
\begin{itemize}
\item $m$ does not divide $n$, in which case $\cyc(h)=(m^k,\lambda)$ for some $\lambda\vdash n-km$, or
\item $m$ divides $n$ and $\cyc(h)=(m^k,\lambda)$ for some $\lambda\vdash m$ (since $n-km=m$),
\item $m$ divides $n$ and $\cyc(h)=(2m, m^{k-1})$; this is the only way that some cycle of $h$ has length greater than $m$, while still having at least $k$ cycles that divide $m$.
\end{itemize}

For $\lambda$ a partition with $\ell$ parts, define the further auxiliary quantity 
$$F_\lambda(\zeta) :=~ \prod_{j=1}^\ell \frac{1}{1-\zeta^{\lambda_j}} ~=~ \prod_{i\geq 1} \frac{1}{(1-\zeta^i)^{c_i}},$$
so that
$$F_h(\zeta) = \lim_{q\to\zeta} ~ \prod_{j=1}^k (1-q^{jm}) \cdot \prod_{j=2}^{k} \frac{1}{1-q^{m}} \cdot
\begin{cases}
\ds \frac{1}{1-q^{2m}} & \text{ if } \cyc(h)=(2m,m^{k-1}) \\~\\
\ds \frac{F_\lambda(\zeta)}{1-q^m} & \text{ otherwise}
\end{cases} $$
In both cases, the $k$ factors of $1-q^m$ in the denominator may be pulled into the first product, yielding the $q$-factorial $[k]!_{q}$ evaluated at $q=\zeta^m=1$. Hence, $F_h(\zeta)$ is either $\frac{k!}{1+\zeta^m}=\frac12k!$ for the exceptional cycle type, otherwise is $k!\,F_\lambda(\zeta)$.

Putting this all together
\begin{align*}
Y(\zeta) = \frac{(1-\zeta)m^k}{|H|} \prod_{i=km+1}^{n-1} (1-\zeta^i) \left( \frac{k!}{2}\right. &\Big|\{h\in H: \cyc(h)=(2m, m^{k-1})\}\Big| \\
& \left. +~ k!\,\sum_{\lambda} \Big|\{h\in H: \cyc(h)=(m^k,\lambda)\}\Big|F_\lambda(\zeta)\right) 
\end{align*}

For the moment, let us assume that $m$ does not divide $n$, so that we may shorten the formula, ignoring the first parenthesized term. Plugging in the definition of $F_\lambda(\zeta)$, we observe that
\begin{align*}
Y(\zeta) &= \frac{1-\zeta}{|H|} \prod_{i=km+1}^{n-1} (1-\zeta^i) \cdot \left(\sum_{\lambda} \,\left|\Conj_{S_n}(m^k,\lambda)\cap H\right|\, m^kk!\,\prod_{i\geq 1}\frac{1}{(1-\zeta^i)^{c_i}}\right) \\
 &= \frac{1-\zeta}{|H|} \prod_{i=km+1}^{n-1} (1-\zeta^i) \cdot \left(\sum_{\lambda} \,\left|\Conj_{S_n}(m^k,\lambda)\cap H\right| \left|Z_{S_n}(m^k,\lambda)\right|\,\prod_{i\geq 1}\frac{1}{((1-\zeta^i)i)^{c_i}c_i!}\right).
\end{align*}
In the last line, we have written $\left|Z_{S_n}(\mu)\right|$ to denote the size of the centralizer of any element in $S_n$ with cycle type $\mu$.

In the case when $m$ divides $n$, note that the initial product simplifies, since in that case it includes all roots of unity except $1$ itself, and thus as argued before evaluates to $m$. Then, performing similar calculations to the above yields:
\begin{align*}
Y(\zeta) = \frac{(1-\zeta)}{|H|} \left( \frac{mk}{4}\right. & \left|\Conj_{S_n}(2m,m^{k-1})\cap H\right| \left|Z_{S_n}(2m,m^{k-1})\right| \\
& \left. + \sum_{\lambda} \,\left|\Conj_{S_n}(m^k,\lambda)\cap H\right| \left|Z_{S_n}(m^k,\lambda)\right|\,\prod_{i\geq 1}\frac{1}{((1-\zeta^i)i)^{c_i}c_i!}\right)
\end{align*}

Comparing this to the desired formula, it would suffice to show that 
$$ \left|\Fix_X(\g)\right| = \frac{1}{|H|}\cdot \left|\text{Conj}_{S_n}(\gamma)\cap H\right| \left|Z_{S_n}(\g)\right| $$
for any $\gamma\in S_n$. In fact, the analogous statement is true for any group, not just $S_n$: see Lemma \ref{fixpts} below.
\end{proof}

In analogy to the symmetric group notation, for any group $G$ and any $g\in G$, write $Z_{G}(\gamma)$ to denote the centralizer of $\gamma$ in $G$.

\begin{lem}
\label{fixpts}
For any finite group $G$, any subgroup $H\leq G$, and any $\g\in G$:
$$ \left|\Fix_{G/H}(\g)\right| = \frac{1}{|H|} \cdot \left|Z_G(\g)\right|\left|\mathrm{Conj}_G(\g) \cap H\right|. $$
\end{lem}

\begin{proof}
Note that any $g\in G$ satisfies $\g gH = gH$ if and only if $g^{-1}\g g \in H$, so the left-hand side is zero if and only if the right side is zero. Suppose that the right-hand side is not zero; in particular, that there exists an element $\eta\in H\cap \text{Conj}_G(\g)$. Notice that $|Z_G(\eta)|=|Z_G(\g)|$ and $|\Fix_{G/H}(\eta)|=|\Fix_{G/H}(\g)|$, so we may assume without loss of generality that $\g\in H$.

We want to show that $|H|\cdot |\Fix_X(\g)| = |Z_G(\g)|\cdot |H\cap \text{Conj}_G(\g)|$, or, since all cosets have the same size $|H|$, we may write the left-hand side as $|\{g\in G: g^{-1}\g g\in H\}|$. To show this equality, we observe that the map $\phi: \{g\in G: g^{-1}\g g\in H\} \to H\cap \text{Conj}_G(\g)$ given by $\phi(g)=g^{-1}\g g$ is surjective, and then it suffices to show that every $\phi^{-1}(h)$ has size $|Z_G(\g)|$. In fact, $\phi^{-1}(h)=gZ_G(\g)$ where $g$ is any element in $\phi^{-1}(h)$, because
$$ (xg^{-1})\g (xg^{-1})^{-1} = \g  \qquad\Longleftrightarrow\qquad g^{-1}\g g = x^{-1}\g x.$$
The left equality states that $x\in gZ_G(\g)$; the right equality states that $\phi(x)=\phi(g)=h$.
\end{proof}

\subsection{Evaluating $\left|\Fix_Y(\tau)\right|$}
\label{fixycalc}

In the previous subsection we wrote $Y(\zeta)$ in terms of $X$-fixpoints, but this is only useful for cyclic sieving if there is some relationship between $X$-fixpoints and $Y$-fixpoints. Fortunately, when $C$ acts freely there is a strong relationship between these two fixspaces. Before describing this relationship, we recall some facts from elementary number theory:
\begin{prop}
\label{numthyfacts}
Let $S_n$ be the symmetric group on $Z_n=\{1,\dots, n\}$, and suppose that $\tau\in N_{S_n}(C)$.
\begin{enumerate}[(a)]
\item \label{defd} There exists unique $d\in(\z/n\z)^\times$ and $r\in\z/n\z$ such that $\tau(x)\equiv dx+r \bmod n$. In particular, $d=1$ if and only if $\tau\in C$.
\item \label{conj} $\tau$ is $S_n$-conjugate to $\tilde\tau:x\mapsto dx+r'$ where $r'$ is the smallest nonnegative integer such that $r'\equiv r \bmod \gcd(n,d-1)$.  
\item \label{gcd2} If $\gcd(n,d-1)$ divides $r$ then $\left|\Fix_{Z_n}(\tau)\right|=\gcd(n,d-1)$, and otherwise $\Fix_{Z_n}(\tau)=\varnothing$.
\end{enumerate}
If moreover $\cyc(\tau)=(m^k, 1,1)$ for some integers $m$ and $k$, then $n$ is even and:
\begin{enumerate}[(a)]
  \setcounter{enumi}{3}
\item \label{fixp} using the notation of (\ref{conj}) above, $r'=0$ and the fixpoints of $\tilde\tau$ are $n$ and $\frac{n}{2}$;
\item \label{2val} if $\frac{n}{4}\in\z$, then $m=2$.
\end{enumerate}
\end{prop}

\begin{proof}
The proofs are routine, and we leave them to the end of this subsection.
\end{proof}

For the arguments that follow, by far the most important fact here is part (\ref{gcd2}). It implies that if there is any $\tau\in N_{S_n}(C)$ with cycle type $(m^k,1,1)$, then $n$ must necessarily be even. In particular, this means that $n$ must in fact be even whenever $(\tau,H)$ is \goodpair\!\, and $\tau$ has two one-cycles.

We also follow up on a remark from the end of Section \ref{gpthy}. Clearly $\tau$ is determined by its output mod $n$ for each input, and so part (\ref{defd}) states that the scaling factor $d$ of elements with the desired cycle type has order $m$ as an element of $(\z/n\z)^\times$, and the translation $r$ must be even if $\cyc(\tau)=(m^k,1,1)$. This necessary condition can be elevated to a sufficient condition if $d^b x\not\equiv x$ for any $x$ except $0$ and perhaps $\frac{n}{2}$, for any $1\leq b<m$. Often this condition is quite restrictive; for instance if $n$ is divisible by $3$, then taking $x=\frac{n}{3}$ shows that $d^b \not\equiv 1 \bmod 3$, and so in particular this forces $m=2$. However, if $n$ is prime, then we can divide through by any nonzero $x$, and hence every nonzero $x$ is in a cycle of the same length $m$.

The main idea of the computation of $|\Fix_Y(\tau)|$ is summarized in the following proposition.

\begin{prop}
\label{surjectiveprop}
Let $(\tau,H)$ be \goodpair. Then writing $m$ for the order of $\tau$, as well as $X=S_n/H$ and $Y=C\backslash S_n/H$, the canonical quotient map $\pi:X\to Y$ restricts to a surjective map $\pi_F: \Fix_{X}(\tau)\to\Fix_{Y}(\tau)$. Indeed, for any $y\in Y$,
$$ |\pi_F^{-1}(y)| = 
\begin{cases}
1&\text{ if }n \equiv 1 \bmod{m},\\
2&\text{ if }n \equiv 2 \bmod{m}.
\end{cases}
$$
\end{prop}

\begin{proof}
First, observe that $\pi_F$ is well-defined, that is, $\pi_F(\Fix_X(\tau))\subseteq \Fix_Y(\tau)$, since for any $\gamma H\in \Fix_X(\tau)$, we have $\tau\cdot C\gamma H = C\tau \gamma H = C\gamma H$.

Now suppose that $CgH\in\Fix_Y(\tau)$, that is, $g\in S_n$ is an element such that $\tau\cdot CgH = CgH$. In particular, this means 
$$\tau gH \subseteq \coprod_{i=1}^n c^i gH,$$
where the cosets on the right are all distinct because $C$ acts freely on $X$. Thus, there must be some $i$ such that $\tau gH = c^igH$. 

Note that all elements of $\pi^{-1}(CgH)$ have the form $c^\xi gH$ for some integer $0\leq \xi<n$. Since $C$ acts freely on $X$, these are distinct for distinct $\xi$. Letting $d$ be the element of $(\z/n\z)^\times$ guaranteed by Proposition \ref{numthyfacts}(\ref{defd}), we may write:
$$\tau (c^{\xi} gH)= c^{d\xi}\tau gH = c^{d\xi+i}gH.$$
Thus $\tau (c^{\xi} gH)=c^{\xi} gH$ if and only if $d\xi+i\equiv\xi \bmod n$. In other words, for any $CgH\in Y$,
$$\left|\pi^{-1}_F(CgH)\right|=\left|\{\xi\in\z/n\z: (d-1)\xi\equiv -i \bmod n\}\right|.$$
We remark that the right-hand side is not \emph{a priori} independent of the element $CgH$, because $i$ generally does depend on $gH$.

In this way, the $n\equiv 1 \bmod m$ case resolves immediately. Since $(\tau,H)$ is \goodpair, we know $\tau$ has a unique fixpoint in $\z/n\z$ and thus $\gcd(n,d-1)=1$ by Proposition \ref{numthyfacts}(\ref{gcd2}). In other words, $d-1$ is invertible mod $n$, and so for any value of $i$ we have the unique solution $\xi\equiv -i(d-1)^{-1}\bmod n$.

The argument for $n\equiv 2 \bmod m$ is similar, but we require a technical prerequisite. By definition, $\tau g\in c^igH$, or equivalently, $g^{-1}c^{-i}\tau g\in H$. Therefore, by Lemma \ref{badcycles} below, $-i$ must not be odd. Therefore, we may divide both sides of the congruence by $2$ and solve. Namely,
$$(d-1)\xi\equiv -i\bmod n \qquad\Longleftrightarrow\qquad\ts \xi \equiv -\frac{i}{2}\left(\frac{d-1}{2}\right)^{-1} \bmod \frac{n}{2},$$
where the left congruence is defined by applying Proposition \ref{numthyfacts}(\ref{gcd2}): since $\gcd(n,d-1)=2$, we have $\frac{d-1}{2}$ invertible mod $\frac{n}{2}$. This solution $\xi$ is unique mod $\frac{n}{2}$, and hence there are precisely two solutions mod $n$, as desired.
\end{proof}

Finally, we prove the required technical lemma to complete the proof of Proposition \ref{surjectiveprop}.

\begin{lem}
\label{badcycles}
Fix an element $g\in S_n$ and let $j$ be an integer. For any $\tau\in N_{S_n}(C)$ such that $\cyc(\tau)=(m^k,1,1)$ for some integers $m$ and $k$, then ($n$ is even and):
$$ \cyc(c^j\tau) = 
\begin{cases}
(m^k,1,1)&\text{if } j \text{ is even},\\
(\ell^{\frac{n}{\ell}}) & \text{for some } \ell | n, \text{ if } j \text{ is odd and } m=2,\\
(m^k, 2) & \text{if } j \text{ is odd and } m>2 \text{ is even},\\
\big((2m)^{\frac{k}{2}}, 2\big) & \text{if } j \text{ is odd and } m \text{ is odd}.
\end{cases}
$$
In particular, if we additionally choose $H\leq S_n$ such that $(\tau,H)$ is \goodpair, then no element conjugate to $c^j\tau$ is contained in $H$ for any odd $j$.
\end{lem}

\begin{proof}
The cycle type of $\tau$ together with Proposition \ref{numthyfacts}(\ref{gcd2}) imply that $\gcd(n,d-1)=2$, and so by Proposition \ref{numthyfacts}(\ref{conj}) we have that the cycle type of $c^j\tau$ depends only on the parity of $j$. In particular, all $c^j\tau$ are conjugate to either $\tau$ or $c\tau$, and so it suffices to compute the cycle type of the latter.

Moreover, we may replace $\tau$ with $\tilde\tau$, and together with Proposition \ref{numthyfacts}(\ref{fixp}), we may say that $\tau(x)=dx$ without loss of generality. 
Because $c\tau(x)\equiv dx+1 \bmod n$,
$$ (c\tau)^b(x) = x \qquad\Longleftrightarrow\qquad \left( d^{b-1} + \cdots + d + 1 \right) \!\big(1+(d-1)x\big) \equiv 0 \bmod n. $$

Note that $m$ is the order of $\tau$ in $S_n$, and thus the order of $d$ in $(\z/n\z)^\times$. Moreover, both $n$ and $d-1$ are even by Proposition \ref{numthyfacts}(\ref{gcd2}), and hence the congruence has no solution unless $b$ is even: the left-hand side would be the product two odd numbers and hence nonzero.

This is enough to resolve case $m=2$, in which $d^2\equiv 1\bmod n$. Using this fact, we have the following congruence mod $n$:
\begin{align*}
\left( d^{b-1} + \cdots + d + 1 \right) \!\big(1+(d-1)x\big)
 &\equiv \ts\frac{b}{2}(d + 1) \big(1+(d-1)x\big) \\
 &\ts\equiv \frac{b}{2}\big((d+1) + (d^2-1)x\big) \\
 &\ts\equiv \frac{b}{2}(d+1).
\end{align*}
This expression is independent of $x$. Thus, letting $\ell$ denote the smallest positive integer $b$ such that $\frac{b}{2}(d+1)\equiv 0\bmod n$, we have shown that every cycle of $c\tau$ has length $\ell$.

We now assume that $m>2$. It thus suffices to show that $c\tau$ has exactly one two-cycle, and every other $x$ is contained in a cycle of length $m$ if $m$ is even, or $2m$ if $m$ is odd. It will be convenient to write the prime factorization $n= 2^{e_0}p_1^{e_1}\cdots p_a^{e_a}$, where the $p_i$ are distinct odd primes and each $e_i$ is a positive integer. By \CRT, we may take $p_0=2$ and then restate the equivalence above as
$$ (c\tau)^b(x) = x \quad\Longleftrightarrow\quad \left( d^{b-1} + \cdots + d + 1 \right) \!\big(1+(d-1)x\big) \equiv 0 \bmod p_i^{e_i} \text{ for all } 0\leq i\leq a. $$

Recall that $(n,d-1)=2$, via Proposition \ref{numthyfacts}(\ref{gcd2}), and thus none of the odd $p_i$ divide $d-1$. In particular, $d-1$ is invertible mod $p_i^{e_i}$ and thus we again apply the \CRT:
\begin{align*}
d^m &\equiv 1\bmod n, \\
(d-1)\!\left( d^{m-1} + \cdots + d + 1 \right) &\equiv 0 \bmod p_i^{e_i}  \qquad\text{for all } 1\leq i\leq a, \\
d^{m-1} + \cdots + d + 1  &\equiv 0 \bmod p_i^{e_i}  \qquad\text{for all } 1\leq i\leq a.
\end{align*}
Thus the desired equivalence $\left( d^{b-1} + \cdots + d + 1 \right) \!\big(1+(d-1)x\big) \equiv 0 \bmod p_i^{e_i}$ holds for the odd primes. For the prime $p_0=2$ we use our assumption that $m>2$. By Proposition \ref{numthyfacts}(\ref{2val}), this means that $\frac{n}{4}\notin\z$ and thus $2^{e_0}=2$. Therefore,
$$d^{m-1} + \cdots + d + 1 \equiv m \bmod 2^{e_0}. $$
So, if $m$ is even then every cycle has length dividing $m$, but if $m$ is odd then every cycle has length dividing $2m$ and (in particular) \emph{no} cycle has length $m$.

We conclude by showing that for any $x\not\equiv -(d-1)^{-1} \bmod \frac{n}{2}$, that $(c\tau)^b(x)\neq x$ for any $b<m$. Before doing so, we make two observations. First, the inverse is well-defined because, as above, $\frac{n}{2}$ is odd, so $\gcd(d-1,\frac{n}{2})=1$ by Proposition \ref{numthyfacts}(\ref{gcd2}). Second, in so doing we will complete the proof: it will guarantee that each such $x$ lies in a $c\tau$-cycle of size at least $m$. Taken together with the previous paragraph, this means that all but two $x$ lie in $c\tau$-cycles of the desired size. Moreover, the other two must form a $2$-cycle, since by Proposition \ref{numthyfacts}(\ref{gcd2}) neither is a fixpoint.

Proposition \ref{numthyfacts}(\ref{fixp}) already gives an analogous statement for $\tau$: for every $1\leq b<m$, all but two $x\in\z/n\z$ are contained in an $m$-cycle, and thus $d^bx\not\equiv x$ for any $x\not\equiv 0\bmod\frac{n}{2}$. In particular, choosing $x_i=\frac{n}{p_i}$ for each $1\leq i\leq a$ shows that:
\begin{align*}
\ts\frac{n}{p_i}(d-1)\left(d^{b-1}+\cdots+d+1\right) &\not\equiv 0 \bmod n \qquad\text{for all } 1\leq i\leq a \\
(d-1)\left(d^{b-1}+\cdots+d+1\right) &\not\equiv 0 \bmod p_i \qquad\text{for all } 1\leq i\leq a
\end{align*}
Hence, $d^{b-1}+\cdots+d+1$ is not divisible by any $p_i$ and so is invertible mod $\frac{n}{2}$. Therefore, we have the following necessary condition:
$$ (c\tau)^b(x) = x \qquad\Longrightarrow\qquad 1+(d-1)x \equiv 0 \bmod \ts\frac{n}{2} $$
As discussed above, $d-1$ is invertible mod $\frac{n}{2}$. Therefore, for any $x\not\equiv -(d-1)^{-1}\bmod\frac{n}{2}$ we conclude that $(c\tau)^b(x)\neq x$, as desired.

\end{proof}

Briefly, we return to prove the number-theoretic facts from Proposition \ref{numthyfacts}:

\begin{proof}[Proof (of Proposition \ref{numthyfacts})] 
Throughout the proof, let $z=\gcd(n,d-1)$.

Part (a): Because $\tau C = C\tau$, there is some unique $0\leq d<n$ such that $\tau c = c^d\tau$. Note that also there is an $e$ such that $\tau c^e= c\tau$ for some $e$, and thus $de\equiv 1\bmod n$, so $d$ is invertible in $\z/n\z$. Additionally,
$$\tau(x+1) = (\tau c)(x) = (c^d\tau)(x) \equiv \tau(x)+d \bmod n.$$
Hence, the fact that $\tau(x)=dx+r$ for some $r$ follows by induction. Finally, plugging in $x=n$, the unique $r\in \z/n\z$ that satisfies the equation is $r=\tau(n)$, or $r=0$ if $\tau(n)=n$.

Part (b): Suppose that $a$ and $b$ are integers such that $an+b(d-1)=z$; these exist by B\'ezout's lemma. Moreover, let $q$ be the integer such that $r = r' + qz$. Then, modulo $n$:
\begin{align*} 
(c^{qb}\tau c^{-qb})(x) &\equiv \tau(x-qb)+qb \\
 &\equiv d(x-qb)+r+qb \\
 &\equiv dx + r - qb(d-1) \\
 &\equiv dx + [r-qz],
\end{align*}
and so $\tau$ is conjugate to $\tilde\tau$, as desired.

Part (c): The number of $Z_n$-fixpoints is the number of 1s in the cycle type, and so we may replace $\tau$ with $\tilde\tau$. Then $x$ is a $Z_n$-fixpoint precisely when $dx+r \equiv x\bmod n$, that is, when $(d-1)x\equiv -r'\bmod n$. The left-hand side is divisible by $z$, and thus there are no solutions to this equation unless $r'$ is divisible by $z$; by definition of $r'$, this happens only if $r'=0$. In this case we may divide through by $z$. That is, the following are equivalent:
\begin{align*} 
(d-1)x &\equiv 0 \bmod n \\
\ts\frac{d-1}{g} \cdot x &\equiv 0 \bmod \ts\frac{n}{z} \\
x &\equiv 0 \bmod \ts\frac{n}{z},
\end{align*}
where the last statement holds because $\frac{d-1}{z}$ is invertible mod $\frac{n}{z}$. Thus the solutions to this congruence, and hence the $Z_n$-fixpoints of $\tau$, are $\{1\cdot\frac{n}{z},2\cdot\frac{n}{z},\dots, z\cdot \frac{n}{z}\}$.

For parts (d) and (e) we now have $z=2$. In particular, note that $n$ must be even.

Part (d): Repeating the proof of part (c) we see that $z|r$ and hence $r'=0$, and moreover the fixpoints are $\{1\cdot\frac{n}{2},2\cdot\frac{n}{2}\}$, as desired.

Part (e): Because $d\in(\z/n\z)^\times$, we know $d$ must be odd, so write $d=2d'+1$. If $\frac{n}{4}\in\z$, we compute
$$\ts\tau^2(\frac{n}{4}) \equiv d^2 \cdot \frac{n}{4} \bmod n \equiv 4(d'^2+d')\frac{n}{4} + \frac{n}{4} \bmod n.$$
Thus, $\frac{n}{4}$ is a fixpoint of $\tau^2$; that is, it is contained in a cycle of $\tau$ whose length divides $2$. From part (d) we know that it is not a fixpoint of $\tau$, and thus it must be in a $2$-cycle. But $\tau$ only has cycles of length 1 and $m$; hence $m=2$.
\end{proof}

\subsection{Completing the Proof of Theorem \ref{technicalcsp}}


We combine the previous two subsections into the following result:

\begin{thm}
\label{technicalcsp-simple}
Suppose that $(\tau, H)$ is \goodpair, and write $m$ for the order of $\tau$. If $H$ additionally avoids the cycle type $(4,2^{\frac{n-4}{2}})$, then $Y(\zeta)=\left|\Fix_Y(\tau)\right|$, where $\zeta$ is a primitive $m^\text{th}$ root of unity, and
$$ Y(q) = \frac{1}{[n]_q} \cdot \frac{ \mathcal H(\c[x]^H ,q) }{ \mathcal H(\c[x]^{S_n},q) }. $$
\end{thm}

\begin{proof}
We begin with the case $\cyc(\tau)=(m^{\frac{n-1}{m}},1)$. The calculation from Proposition \ref{moliencalc} simplifies considerably since there is only one $\l$ that partitions $n-km$, namely $\l=(1)$. Therefore, if $\zeta$ is a primitive $m^\text{th}$ root of unity, then
$$ Y(\zeta) = (1-\zeta) \frac{\left|\Fix_X(m^k,1)\right|}{1-\zeta} = \left|\Fix_X(\tau)\right|.$$
Thus, by Proposition \ref{surjectiveprop} we have $Y(\zeta)=\left|\Fix_Y(\tau)\right|$, as desired.

On the other hand, if $\cyc(\tau) = (m^{\frac{n-1}{m}},1,1)$, Proposition \ref{moliencalc} simplifies similarly: there are now two $\l$ that partition $n-km$, namely $\l=(1,1)$ and $\l=(2)$. Therefore, if $\zeta$ is a primitive $m^\text{th}$ root of unity, then 
$$ Y(\zeta)= \left\{
\begin{array}{ll}
\frac12\left|\Fix_X(2^k,1,1)\right|+k\left|\Fix_X(4,2^{k-1})\right| & \text{ if } m=2 \\
\frac12\left|\Fix_X(m^k,1,1)\right| + \left(\frac{1-\zeta}{1+\zeta}\right)\left|\Fix_X(m^k,2)\right| & \text{ if } m>2
\end{array}
\right. .$$
In either case, Lemma \ref{fixpts} and the conditions on $H$ imply that the second term vanishes and $Y(\zeta)=\frac12|\Fix_X(\tau)|$. Thus by Proposition \ref{surjectiveprop} we have $Y(\zeta)=\left|\Fix_Y(\tau)\right|$, as desired.
\end{proof}

\begin{remark}
As mentioned before, trying to extend this argument to $n\equiv 3\bmod m$ is more troublesome. For simplicity let us suppose that $m>3$, then we may see the difficulty by using Proposition \ref{moliencalc} again:
$$Y(\zeta) = \ts\frac16\left(
\frac{2(1-\zeta)(1-\zeta^2)}{1+\zeta+\zeta^2}\left|\Fix_X(m^k,3)\right| + 
3(1-\zeta)\left|\Fix_X(m^k,2,1)\right| +
(1+\zeta)\left|\Fix_X(m^k,1,1,1)\right|
\right).$$
None of these coefficients are rational, and so if $Y(\zeta)$ is to evaluate to a positive integer, it must have contributions from multiple terms. Unlike for the $n\equiv 2 \bmod m$ case, we cannot simply exclude the cycle type giving complex contribution and focus on the only $\tau$ remaining.
\end{remark}

This is nearly all of Theorem \ref{technicalcsp}; to complete the proof, we must compute $Y(\zeta)$ at non-primitive roots of unity. Recall that the conditions on $\tau$ and $H$ of Theorem \ref{technicalcsp} are equivalent to the fact that $(\tau^b, H)$ is \goodpair\!\, for all $1\leq b\leq m-1$ and also $H$ avoids $(4,2^{\frac{n-4}{2}})$. Thus, we may apply Theorem \ref{technicalcsp-simple} to all powers $\tau^b$ of $\tau$, in which case the corresponding order will be $\frac{m}{\gcd(m,b)}$ and hence the corresponding root of unity may be chosen to be $\zeta^b$. 

Hence we have shown that $\left|\Fix_Y(\tau^b)\right| = Y(\zeta^b)$ for all $1\leq b\leq m-1$. Since the $b=0$ case is straightforward, this completes the proof of cyclic sieving.
\qed

\section{Secondary Cyclic Sieving}
\label{futurework}

We conclude by observing that Theorem \ref{technicalcsp} appears to be an example of a more general phenomenon.

\begin{defn}
Suppose that a group $G$ acts on a set $X$, and $(X,X(q),C)$ exhibits the cyclic sieving phenomenon for some polynomial $X(q)$ and cyclic subgroup $C\leq G$ which acts freely. Define $Y=C\backslash X$, and the action of $N_G(C)$ on $Y$ via $\tau \cdot Cx = C(\tau \cdot x)$, and the polynomial $Y(q)=\frac{1}{[n]_q}X(q)$. Then for any $\tau$ such that $(Y,Y(q),\<\tau\>)$ exhibits the cyclic sieving phenomenon, we say that $\tau$ \textbf{exhibits a secondary cylic sieving phenomenon} with respect to $(X,X(q),C)$.
\end{defn}

\begin{remark}
As before, $Y(q)$ is a polynomial in $\z[q]$ because of Proposition \ref{freeactpoly}.
\end{remark}

In particular, Theorem \ref{braceletcount} can be reformutated as stating that if $\gcd(\a)=1$ and $W(\alpha)$ is the collection of words having exactly $\a_i$ occurrences of the letter $i$, then the action of reflection exhibits a secondary cyclic sieving phenomenon with respect to the triple $(W(\alpha), {n\brack\a}_q,C)$. Similarly, Theorem \ref{technicalcsp} can be understood as describing some sufficient conditions for an element $\tau\in S_n$ to exhibit a secondary cyclic sieving phenomenon for the triple $(S_n/H, X(q), C)$ for the polynomial $X(q)$ as given in (\ref{coset-polynomial}).

As a different example, consider the set $X_k$ of noncrossing partitions $\{1,\dots, n\}$ with $k$ blocks. Recall that a set partition of $\{1,\dots, n\}$ is called \textbf{noncrossing} if for any four numbers $i<j<p<q$ such that $i$ and $p$ are in the same block, and $j$ and $q$ are in the same block, then in fact all four are in the same block. Noncrossing partitions admit a geometric action by the subgroup $D_n\is\<c,\tau\>$ of $S_n$, where $c$ and $\tau$ are the same elements that act on $W(\alpha)$ as described in Section \ref{necklacecsp}.

The elements of $X_k$ are counted by the \textbf{Narayana numbers} $N(n,k)$, which have a product formula that suggests a $q$-analogue:

$$N(n,k) = \frac{1}{n}\binom{n}{k}\binom{n}{k+1}; \qquad\qquad X_k(q) = \frac{1}{[n]_q}\pvec{n}{k}_{q}\pvec{n}{k-1}_{q}.$$

It was shown in \cite[\S 7]{ReiStaWhi2004} that $(X_k,X_k(q), \<c\>)$ exhibits the cyclic sieving phenomenon. Moreover, we can check that $\<c\>$ acts freely whenever $\gcd(n,k)=\gcd(n,k-1)=1$. In particular, this implies that $n$ must be odd, and thus $[n]_q$ evaluates to $1$ at $q=-1$. Thus:
$$ Y_k(-1) = \left. \frac{1}{[n]_q^2}\pvec{n}{k}_{q}\pvec{n}{k-1}_{q} \right|_{q=-1} = \left. \frac{1}{[n]_q}\pvec{n}{k}_{q}\pvec{n}{k-1}_{q} \right|_{q=-1} = X_k(-1).$$

In \cite[\S 3.2]{Din2016} Ding computes the number of $\tau$-fixed elements of $X_k$ to be $X_k(-1)$, and hence $Y_k(-1)$. Also, in \cite[\S 4]{CalSmi2005}, Callan and Smiley show the surprising fact that the number of $\tau$-fixed elements of $Y_k$ is the same as the number of $\tau$-fixed elements of $X_k$. We thus conclude that $Y(-1)=\#\{y\in Y_k: \tau(y)=y\}$. Hence, $(Y_k,Y_k(q),\<\tau\>)$ exhibits the cyclic sieving phenomenon, or in other words, $\tau$ exhibits a secondary cyclic sieving phenomenon with respect to $(X_k,X_k(q),\<c\>)$.

The cyclic sieving literature is extensive (see, for instance, \cite{BodRho2016}, \cite{Gor2019}, \cite{ReiSom2018}, \cite{SheWen2020}, \cite{Thi2017}) and so we conclude with the following natural question:

\begin{question}
To what extent do previously known cyclic sieving results admit secondary cyclic sieving phenomena?
\end{question}

\section{Acknowledgments}

This work was supported by NSF grant DMS-1601961. The author also wishes to thank an anonymous referee who left generous comments and pointed out an error in an earlier version of the paper.

\bibliography{masterbib}{}
\bibliographystyle{abbrv}

\end{document}